\documentclass[12pt]{amsart}
\usepackage{graphicx}
\usepackage{amssymb}
\usepackage{bm}

\newtheorem{theorem}{Theorem}
\newtheorem{lemma}[theorem]{Lemma}
\newtheorem{proposition}[theorem]{Proposition}
\newtheorem{corollary}[theorem]{Corollary}

\theoremstyle{definition}
\newtheorem{definition}[theorem]{Definition}

\theoremstyle{remark}
\newtheorem{remark}[theorem]{Remark}
\newtheorem{example}[theorem]{Example}
\newtheorem{notation}[theorem]{Notation}

\newcommand{\bl}[1]{\overline{#1}}

\newcommand{\mfb}{\mathfrak{b}}

\newcommand{\bj}{{\boldsymbol{j}}}
\newcommand{\balpha}{\boldsymbol{\alpha}}

\newcommand{\cU}{\mathcal{U}}
\newcommand{\cB}{\mathcal{B}}
\newcommand{\cC}{\mathcal{C}}

\newcommand{\tr}{\operatorname{tr}}
\newcommand{\E}{\operatorname{E}}

\newcommand{\bC}{\mathbb{C}}

\newcommand{\bR}{\mathbb{R}}

\newcommand{\cO}{\mathcal{O}}

\let\phi=\varphi

\newcommand{\centre}{\mathaccent"7017}

\newcommand{\Tr}{\operatorname{Tr}}
\newcommand{\cov}{\mathrm{cov}}

\newcommand{\Wg}{\operatorname{Wg}}
\newcommand{\moeb}{\operatorname{M\ddot{o}b}}

\def\CC{\mathbb{C}}

\newcommand{\cA}{\mathcal{A}}

\newcommand{\cP}{\mathcal{P}}

\theoremstyle{definition}

\newcommand{\ab}{\allowbreak}

\def\bR{\mathbb{R}}

\newcommand{\cc}{k}         

\newcommand{\bi}{{\boldsymbol{i}}}

\newcommand{\lra}{\longrightarrow}

\newcommand{\fU}{\mathfrak{U}}

\newcommand{\pn}{ P_2 ( n ) }

\newcommand{\pnn}{ P_2( \pm  n ) }

\newcommand{\cE}{\mathcal{E}}

\title[{Freeness and the Transpose}] {Freeness and The Transposes
  of Unitarily Invariant Random Matrices}

\author[mingo]{James A. Mingo$^{(*)}$} \address{Department
  of Mathematics and Statistics, Queen's University, Jeffery
  Hall, Kingston, Ontario, K7L 3N6, Canada}

\email{mingo@mast.queensu.ca} 

\thanks{$^*$ Research supported by a Discovery Grant from
  the Natural Sciences and Engineering Research Council of
  Canada}

\author[popa]{Mihai Popa$^{( * )(* *)}$ } \address{The
  University of Texas at San Antonio, Department of
  Mathematics, One UTSA Circle, San Antonio, Texas 78249,
  and \newline ${}\hspace{.4cm}{}$ Institute of Mathematics
  ``Simion Stoilow'' of the Romanian Academy, P.O. Box
  1-764, Bucharest, RO-70700, Romania }

\email{Mihai.Popa@utsa.edu} 

\thanks{$^{ * * }$ Research supported by the Natural Science
  Foundation of China Grant No. 11150110456, and the
  Romanian National Authority for Scientific Research, CNCS
  UEFISCDI, Project Number PN-II-ID-PCE-2011-3-0119}

\begin{document}

\begin{abstract}
 We show that real second order freeness appears in the
 study of Haar unitary and unitarily invariant random
 matrices when transposes are also considered. In particular
 we obtain the unexpected result that a unitarily invariant
 random matrix will be asymptotically free from its
 transpose.
\end{abstract}

\maketitle

\section{Introduction}

Free independence, introduced by D.-V. Voiculescu, is an
analogue of classical independence suited for non-commuting
random variables. If two random variables $ a $ and $ b $
are freely independent, there is a universal rule, independent
of $ a $ and $ b $, for computing the mixed moments of $ a $
and $ b $ from just the moments of $ a $ and the moments of
$ b $.  This has been a very useful tool in random
matrix theory as independent ensembles are asymptotically
free, provided that at least one of them is unitarily, or
orthogonally, invariant; see \cite{v1}, \cite{v2},
\cite{agz}, \cite{mss}.

The study of fluctuations of the distribution of eigenvalues
around their mean, for various ensembles of random matrices,
goes back to the work of Dyson (see \cite{d1},
\cite{d2}). These fluctuations can be often described by the
use of fluctuation moments. Second order freeness was
developed by Mingo and Speicher (see \cite{ms}, \cite{mst})
to give a universal rule for computing mixed fluctuation
moments in terms of the individual fluctuation
moments. Random variables are second order free when this
universal rule holds. It was shown, in \cite{mss} that, as
in the case of Voiculescu's freeness, independent ensembles
are asymptotically second order free provided that
at least one of them is unitarily invariant.

In the case of second order freeness differences in the
symmetries of the matrices become apparent.  Emily
Redelmeier \cite{r} showed that for the {\sc goe}, real
Wishart matrices, and real Ginibre matrices a different
universal rule was needed. In the first order case,
i.e.~Voiculescu's freeness, no difference between real and
complex matrices is observed. In the second order case new
terms appeared. Redelmeier called this new universal rule
\textit{real second order freeness}. In \cite{mp} the
authors showed that independent and orthogonally invariant
matrices are asymptotically real second order free. Since
the orthogonal group is a subgroup of the unitary group,
both rules apply in the case of unitarily invariant
ensembles.

In order to resolve the discrepancy between the two rules
the authors were led to examine the role played by the
transpose in random matrix theory.  This lead to the
surprising fact that many ensembles of random matrices are
asymptotically free from their transposes. The goal of this
paper is to show that for unitarily invariant ensembles the
transpose does produce asymptotic freeness (see Corollary
\ref{cor:t-free} and Proposition \ref{prop:31}).

Let us illustrate this with a simple example, see Figure
\ref{fig-1}. Let $U$ be a $N \times N$ Haar distributed
random unitary matrix. Let $U^*$ denote the conjugate
transpose of $U$, and $U^t$ the transpose of $U$. The
eigenvalue distribution of $U + U^*$ converges to the
arcsine law, $\mu_1$, on $[2, -2]$ (i.e. with density
$1/(\pi\sqrt{ 4 - t^2})$). If we take the free additive
convolution of $\mu_1$ with itself we get $\mu_2 = \mu_1
\boxplus \mu_1$, the Kesten-McKay law with two degrees of
freedom. Since the eigenvalue distribution of $U + U^* + (U
+ U^*)^t$ appears to converge to $\mu_2$ are led to expect
that $U + U^*$ is asymptotically free from its transpose. We
shall prove a much more general result, namely one needs the
ensemble to be unitarily invariant.
\begin{figure}
\hfill
\includegraphics[scale=0.5]{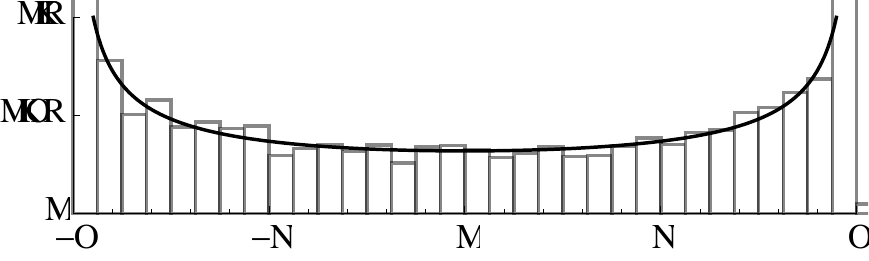}\hfill
\includegraphics[scale=0.4]{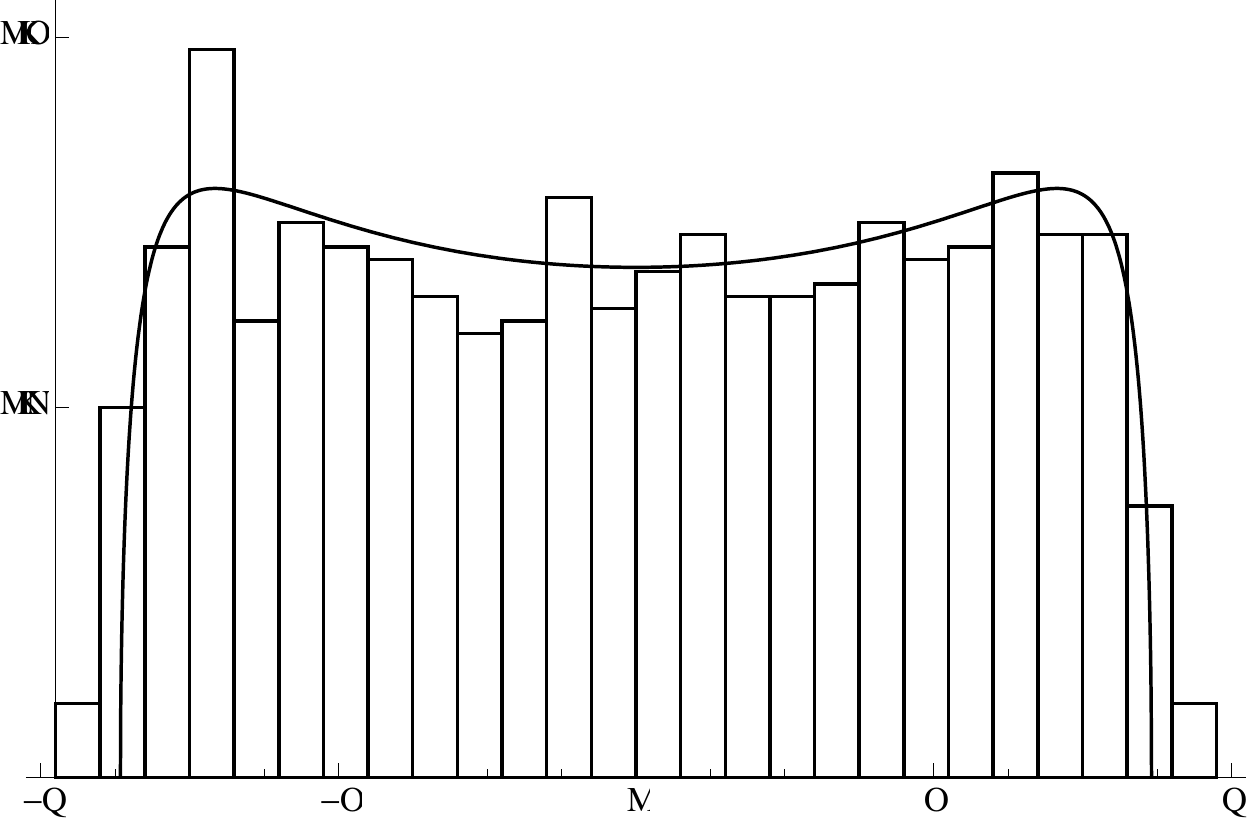}\hfill\hbox{}
\caption{\small \label{fig-1}
 On the left we have sampled $U + U^*$ with
  $U$ a Haar distributed random unitary matrix. One can see
 that the eigenvalue distribution converges to the arcsine
 law (solid line) on the interval $[2, -2]$. On the right
 we have the distribution of the eigenvalues of $U + U^* +
 (U + U^*)^t$. converging to the Kesten-Mckay law with two
  degrees of freedom (solid line).}
\end{figure}
Besides this introduction, the paper contains 3 more
sections: Section 2 states definitions and some preliminary
results, Section 3 addresses the relation between Haar
unitary random matrices and their transposes and Section 4
contains the main results, regarding ensembles of unitarily
invariant random matrices.

\section{Preliminaries}

\subsection{Notations}\label{notations} 
Let $ A $ be a $ N \times N $ matrix, where $ N \geq 2 $. We
will denote its non-normalized trace by $ \Tr(A) $, and its
normalized trace by $ \tr( A ) $, i.e. $ \tr( A) =
\frac{1}{N} \Tr( A) $.  Also we will denote by by $ A^{t} $
the transpose of the matrix $ A $.  We will use the notation
$ A^{( - 1 ) } = A^{ t } $ and $ A ^{ ( 1 )} = A $.  We need
to extend this notation to allow taking entry-wise complex
conjugates. Let $\bl{A} = (A^*)^t$ and for $(\epsilon, \eta)
\in \{-1, 1\}^2$ let $A^{(\epsilon, \eta)}$ be defined as
follows:

\medskip
$\hfill
A^{(1,1)} = A, \hfill
A^{(-1,1)} = A^t, \hfill
A^{(1,-1)} = \bl{A}, \hfill
A^{(-1,-1)} = A^*.
\hfill\hbox{}$
\medskip

\noindent 
The first variable indicates whether or not a transpose is
taken and the second whether or not the complex conjugate of
each entry is taken.  We shall use the notation that
$-(\epsilon, \eta) = (-\epsilon, -\eta)$.  Note that then
$(A^{(\epsilon, \eta)})^* = A^{-(\epsilon, \eta)}$. For a
complex number $z$ let $z^{(1)} = z$ and $z^{(-1)} =
\bl{z}$. Note that this is a slight abuse of notation
because a complex number can also be regarded as a $1 \times
1 $ matrix. However as we are mainly interested large
matrices this will not cause any confusion. Then the $(i,j)$
entry of $A^{(\epsilon, \eta)}$ is $a_{ij}^{(\eta)}$ if
$\epsilon = 1$ and $a_{ji}^{(\eta)}$ if $\epsilon = -1$.

For a positive integer $n$, let $ [ n ] $ be the ordered set
$ \{ 1, 2, \dots, n \} $ and $ \cP ( n ) $ be the set of all
partitions of $ [ n ] $.  By $k_r(X_1, \dots, X_r)$ we mean
the $r^{th}$ classical cumulant of the random variables
$X_1, X_2, \dots, X_r$. Recall from [{\sc ns}, p. 191] that
these are defined implicitly by the moment cumulant relation
\begin{equation}\label{eq:moment-cumulant}
\E(X_1 \cdots X_r) = \sum_{\cU \in \cP(r)} k_\cU(X_1, \dots, X_r)
\end{equation}
where $ \cU \in \cP ( r ) $ and  
\[
  \cc_\cU ( X_1, \dots, X_r ) = \prod_{\substack{B\in
      \cU\\ B = \{ i_1, \dots, i_p\} } } \cc_p ( X_{ i_1 },
  \dots, X_{ i_p } ).
\]

If $\cA$ is a unital algebra and $\cA_1, \cA_2, \dots,
\cA_s$ are unital subalgebras of $\cA$, we will say that a
$n$-tuple $ (a_1, \dots, a_n) $ of elements from $\cA_1 \cup
\cdots \cup \cA_s$ is \emph{alternating} if, for $1 \leq i
\leq n -1 $ we have that $ a_i \in\cA_{j_i} $ with $j_i \neq
j_{i+ 1}$.  If in addition $j_n \not = j_1$, we say that the
$n$-tuple is \emph{cyclically alternating}. If $\phi: \cA
\rightarrow \bC$ is a state, i.e. $\phi$ is linear and
$\phi(1) = 1$, we say $a \in \cA$ is \textit{centered} if
$\phi(a) = 0$.
  
Finally, by $ \mathbb{C}\langle x_1, \dots, x_r \rangle $ we
will denote the algebra of polynomials with complex
coefficients in the non-commuting variables $ x_1, \dots,
x_r $.

\subsection{Second order distributions and second order freeness}${}$

We consider a \textit{non-commutative probability space}
$(\cA, \phi)$, that is $\cA$ is a unital algebra over $\bC$
and a state $\phi: \cA \rightarrow \bC$. For $x_1, \dots,
x_n \in \cA$ the mixed moments of $x_1, \dots, x_n$ are the
values $\phi(x_{i_1} \cdots x_{i_k})$ where $1 \leq i_l \leq
n$ for $1 \leq l \leq k$. The elements $x_1, \dots, x_n$ are
\textit{free} if the the unital subalgebras $\cA_1, \dots,
\cA_n$ are free where $\cA_i$ is the unital subalgebra
generated by $x_i$. Unital subalgebras $\cA_1, \dots, \cA_s$
of $\cA$ are \textit{free} if whenever we have centered and
alternating elements $a_1, \dots, a_n \in \cA_1 \cup \cdots
\cup \cA_s$, we have $\phi(a_1 \cdots a_n) = 0$. This gives
a universal rule for computing the mixed moments
$\phi(x_{i_1} \cdots x_{i_k})$ in terms of the moments of
each $x_i$. See \cite[Examples 5.15]{ns}.  We shall often
call this freeness property (introduced by of Voiculescu,
\cite{v1}) \textit{first order freeness} to distinguish it
from that of second order freeness given below.

A \textit{second order non-commutative probability space},
$(\cA, \phi, \phi_2)$, is a non-commuta\-tive probability
space $(\cA, \phi)$ with extra structure. We first require
that $\phi$ be a tracial state, i.e. $\phi(ab) = \phi(ba)$
and second the existence of a bilinear functional $\phi_2$
that gives the fluctuation moments. This means $\phi_2: \cA
\times \cA \rightarrow \bC$ is a bilinear functional,
tracial in each variable, is such that $\phi_2(1, a) =
\phi_2(a, 1) = 0$ for all $a \in \cA$. For $x_1, \dots, x_s
\in \cA$ the mixed fluctuation moments of $x_1, \dots, x_s$
are the values $\phi_2(x_{i_1}\cdots x_{i_m}, x_{j_1} \cdots
x_{j_n})$ with $1 \leq i_k \leq s$ and $1 \leq j_l \leq
s$. Second order freeness gives a universal rule for
computing the mixed fluctuation moments in terms of the
moments and the fluctuation moments of each $x_i$. See
\cite[Def. 6.3]{ms}, where $\phi_2$ is denoted by $\rho$.

As with first order freeness, we say that random variables
$x_1, \dots, x_s$ are second order free if the unital
subalgebras $\cA_1, \dots, \cA_s$ they generate are second
order free. Unital subalgebras $\cA_1, \dots, \cA_s$ are
second order free if they are free of first order, and
satisfy the following property. Given centered
elements $a_1, \dots a_m, b_1, \dots b_n \in \cA_1 \cup
\cdots \cup \cA_s$ such that both $a_1, \dots, a_m$ and
$b_1, \dots, b_n$ are cyclically alternating we have a rule
for computing the fluctuation moments $\phi_2(a_1 \cdots
a_m, b_1 \cdots b_n)$ in terms of the moments
$\phi(a_ib_j)$. Namely for $m,n \geq 1$ with $m + n > 2$
\begin{equation}\label{eq:secondorderfreeness}
\phi_2(a_1 \cdots a_m, b_1 \cdots b_n) = \delta_{m,n}
\sum_{k=1}^n \prod_{i=1}^n \phi(a_i b_{k-i})
\end{equation}
where $k-i$ is interpreted modulo $n$. When $m = n = 1$ we
also require $\phi_2(a_1, b_1) = 0$ if $k_1 \not = l_1$. The
expression on the right hand side of
(\ref{eq:secondorderfreeness}) has a simple interpretation
in terms of \textit{spoke diagrams}. We put the points
$a_1$, $a_2, \dots, a_m$ around a circle in clockwise order
and the points $b_1$, $b_2, \dots, b_n$ in counter-clockwise
order around an inner circle (see
Fig. \ref{fig-2}). Equation (\ref{eq:secondorderfreeness})
says that the universal rule for computing $\phi_2(a_1a_2
\cdots a_m, b_1b_2\cdots b_n)$ is the sum over all possible
spoke diagrams. If $m \not = n$ then no spoke diagrams are
possible and the sum is $0$. If $m = n =1$ then there is no
reduction unless $a_1$ and $b_1$ come from different
subalgebras. When $m = n \geq 2$, the fluctuation moments
can be computed from certain first order moments.  Note that
equation (\ref{eq:secondorderfreeness}) also means that if a
spoke connects an $a_i$ to a $b_j$ and $a_i$ and $b_j$ do
not come from the same subalgebra then $\phi(a_ib_j) = 0$,
since we have assumed that $a_1, \dots, a_m$ and $b_1,
\dots, b_n$ are centered. Hence in
(\ref{eq:secondorderfreeness}) only spoke diagrams which
connect elements from the same subalgebra make a non-zero
contribution.  When $m = n = 3$ the three terms
corresponding to the three \textit{spoke} diagrams are shown
in Figure \ref{fig-2}.

%
\setbox1=\hbox{\small
$\phi(a_1b_1) \phi(a_2b_3) \phi(a_3 b_2)$}
\setbox2=\hbox{\small
$\phi(a_1b_2) \phi(a_2b_1) \phi(a_3 b_3)$}
\setbox3=\hbox{\small
$\phi(a_1b_3) \phi(a_2b_2) \phi(a_3 b_1)$}

\begin{figure}
\noindent 
\vbox{\hsize\wd1\parindent0pt%
\hfill\includegraphics{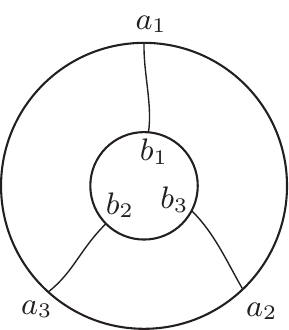}\hfill\hbox{}\\
\box1} \hfill
\vbox{\hsize\wd2\parindent0pt%
\hfill\includegraphics{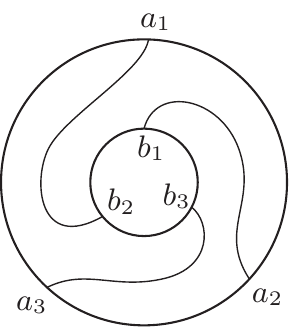}\hfill\hbox{}\\
\box2} \hfill
\vbox{\hsize\wd3\parindent0pt%
\hfill\includegraphics{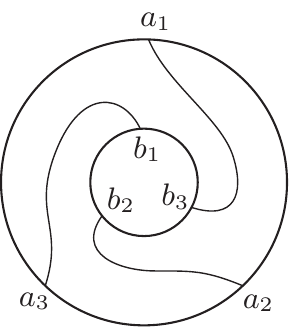}\hfill\hbox{}\\
\box3}
\caption{\small The three spoke diagrams needed by
 equation~(\ref{eq:secondorderfreeness}) to evaluate
  $\phi_2(a_1a_2a_3, b_1b_2b_3)$.}\label{fig-2}
\end{figure}

A \textit{real second order non-commutative probability
  space}, $(\cA, \phi, \phi_2, {}t)$, is a non-commuta\-tive
probability space $(\cA, \phi, \phi_2)$ with a transpose
$a\mapsto a^t$ (a linear map with $(ab)^t = b^t a^t$). We
assume that $\phi(a^t) = \phi(a)$ and $\phi_2(a^t, b) =
\phi_2(a, b^t) = \phi_2(a, b)$.

We say that random variables $x_1, \dots, x_s$ in a real
second order non-commutative probability space, $(\cA, \phi,
\phi_2, t)$, are \textit{real second order free} if the
unital subalgebras $\cA_1, \dots, \cA_s$ they generate are
real second order free. Unital subalgebras $\cA_1, \dots,
\cA_s$ are \textit{real second order free} if they are free
of first order and satisfy the following extension of
equation (\ref{eq:secondorderfreeness}). Given centered
elements $a_1, \dots a_m, b_1, \dots b_n \in \cA_1 \cup
\cdots \cup \cA_s$ such that both $a_1, \dots, a_m$ and
$b_1, \dots, b_n$ are cyclically alternating we have a rule
for computing the fluctuation moments $\phi_2(a_1 \cdots
a_m, b_1 \cdots b_n)$ in terms of the moments $\phi(a_ib_j)$
and $\phi(a_ib_j^t)$. Namely for $m,n \geq 1$ such that
 $m + n >2$
\begin{eqnarray}\label{eq:realsecondorderfreeness}\lefteqn{
\phi_2(a_1 \cdots a_m, b_1 \cdots b_n) } \\
& = & 
\delta_{m,n} \sum_{k=1}^n \prod_{i=1}^n \phi(a_i b_{k-i}) +
\delta_{m,n} \sum_{k=1}^n \prod_{i=1}^n \phi(a_i
b_{k+i}^t)\qquad \notag
\end{eqnarray}
where $k+i$ and $k-i$ are interpreted modulo $n$. When $m =
n = 1$ we also require $\phi_2(a_1, b_1) = 0$ if $k_1 \not =
l_1$. The right hand side of
(\ref{eq:realsecondorderfreeness}) also has an
interpretation in terms of spoke diagrams. The first term is
the same as for second order freeness. The second term on
the right hand side also has an interpretation in terms of
\textit{reversed spoke diagrams}.  We put the points $a_1$,
$a_2, \dots, a_m$ around a circle in clockwise order and the
points $b_1^t$, $b_2^t, \dots, b_n^t$ in clockwise order
around an inner circle (see Fig. \ref{fig-3}). Equation
(\ref{eq:realsecondorderfreeness}) says that the universal
rule for computing $\phi_2(a_1a_2\cdots a_m, b_1b_2\cdots
b_n)$ is the sum of terms corresponding to the
\textit{spoke} diagrams in Figure \ref{fig-2} \textit{plus}
the additional terms coming from the reversed spoke diagrams
in Fig. \ref{fig-3}.

When these universal rules only hold asymptotically we have
what is called \textit{asymptotic freeness}.

\setbox1=\hbox{\small
$\phi(a_1b_1^t) \phi(a_2b_2^t) \phi(a_3 b_3^t)$}
\setbox2=\hbox{\small
$\phi(a_1b_3^t) \phi(a_2b_1^t) \phi(a_3 b_2^t)$}
\setbox3=\hbox{\small
$\phi(a_1b_2^t) \phi(a_2b_3^t) \phi(a_3 b_1^t)$}

\begin{figure}
\noindent
\vbox{\hsize\wd1\parindent0pt%
\hfill\includegraphics{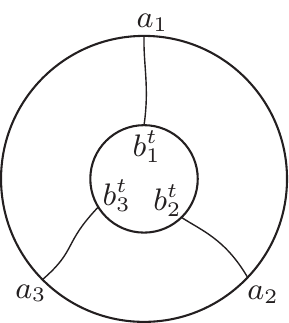}\hfill\hbox{}\\
\box1} \hfill
\vbox{\hsize\wd2\parindent0pt%
\hfill\includegraphics{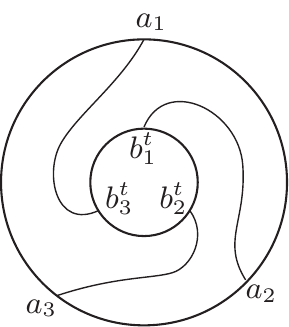}\hfill\hbox{}\\
\box2} \hfill
\vbox{\hsize\wd3\parindent0pt%
\hfill\includegraphics{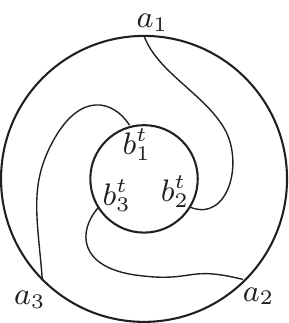}\hfill\hbox{}\\
\box3}
\caption{\small The three reversed spoke diagrams needed
  by equation~(\ref{eq:realsecondorderfreeness}) to evaluate
  $\phi_2(a_1a_2a_3, b_1b_2b_3)$ in the case of real second
 order freeness.}\label{fig-3}
\end{figure}

\begin{definition} \label{defn:1}

Suppose that for each $ N $ we have a set of $ s $ $ N
\times N $ random matrices $ \{ A_{ N, 1 }, A_{ N, 2 },
\dots A_{ N, s } \} $ and denote by $ \cA $ the ensemble $
\coprod_{ N \geq 1 } \{ A_{ N, 1 }, A_{ N, 2 }, \dots A_{ N,
  s } \} $ and by $\cA^t$ the ensemble $\coprod_{ N \geq 1 }
\{ A^t_{ N, 1 }, A^t_{ N, 2 },\ab \dots A^t_{ N, s } \} $.

\medskip\noindent(\textbf{1}) We say that the ensemble $ \cA
$ has the \emph{bounded cumulants property} if for all
polynomials $ p_1, p_2, p_3, \dots $ from the algebra $
\mathbb{C} \langle x_1, x_2, \dots, x_s \rangle $ and for
integers $ k \geq 1 $, denoting
\[  
W_{ N, k } = p_k ( A_{ N, 1 }, A_{ N, 2 }, \dots A_{ N, s }
)
\]
we have that:
\begin{enumerate}

\item[(\textit{i})]  
$\displaystyle \sup_{N}| \E( \tr( W_{N,i} )
  ) | < \infty $, for all $i$

\item[(\textit{ii})] $\displaystyle \sup_{N}| \cc_r ( \Tr (
  W_{N,1} ), \Tr ( W_{N,2} ), \dots, \Tr ( W_{N,r} ) ) | <
  \infty $ for all $ r \geq 2 $.
\end{enumerate}
If the ensemble $\cA \cup \cA^t$ has the bounded cumulants property,
we say that $\cA$ has  the \textit{$ t $-bounded cumulants property}.

\medskip\noindent
(\textbf{2}) We say that the ensemble $ \cA $ has a
\emph{second order limit distribution}, if, in the
definition above, conditions (\textit{i}) and (\textit{ii})
are replaced by
\begin{enumerate} 

\item[(\textit{i}$^\prime$)] 
$\displaystyle \lim_{ N \rightarrow \infty }| \E( \tr(
  W_{N,i} ) ) | $\ exists and it is finite for all $i$.

\item[(\textit{ii}$'$)]
$\displaystyle \lim_{N \rightarrow \infty} \cc_2 ( \Tr(
  W_{N,1}), \Tr(W_{N,2} )) $\ exists and it is finite.

\item[(\textit{iii}$'$)] 
$\displaystyle \lim_{ N \rightarrow \infty } \cc_r (
  \Tr(W_{N,1}), \ldots, \Tr(W_{N,r}) ) = 0 $\ for all $ r
  \geq 3 $.

\end{enumerate}
If $\cA \cup \cA^t$ has a second order limit distribution,
we say that $\cA$ has a \textit{second order limit
  $t$-distribution}. If only the condition ( $ i ^ \prime $ ) is satisfied, 
then  the ensemble $ \cA $ will be said to have a \emph{limit distribution}. 

\end{definition}

\begin{example}  
It is easy to construct ensembles of random matrices with
a second order limit distribution, but no second order limit
$t $-distribution (see also \cite[\S5]{cs}). Consider $ \cA
= \{ A_N \}_{ N \geq1} $, with

\[ 
A_N = [ i^{ k } \cdot \delta_{ l - k, 1 } ]_{ k, l =1 }^ N,
\]
i.e. all entries of $ A_N $ are zero, except for the $ ( k,
k+1) $-entries (on the diagonal above the main one), which
are equal to $ i^{ k } $ (here $ i $ is the complex number,
$ i^2 = -1 $ ).  Then $ \tr( A_N^p ) = 0$ for all positive
integers $ p$, but $ \lim_{ N \rightarrow \infty } \tr( A_N
\cdot A_N^t ) $ does not exist.
\end{example}

Throughout the rest of the paper, since all matrices
appearing in an expression will always have the same
dimension, we will omit the first index, $N$, when writing
matrices from an ensemble, i.e. we will write $ A_1, A_2,
A_2 $ for $ A_{N, 1}, A_{N, 2}, A_{N, 3} $.

\begin{definition}\label{defin:03}
Let $ \cA = \{ A_1, \dots, A_s \}_{ N \geq 1 } $ and $ \cB
=\{ B_1, \ldots, B_p \}_{ N \geq 1 } $ be two ensembles of
random matrices.

(\textit{i}) We say that $ \cA $ and $ \cB$ are\emph{
  asymptotically free} if $ \cA \cup \cB $ satisfies
condition (\textit{i}$'$) from Definition \ref{defn:1}, and
if for any positive integer $n $ and any centered
alternating $n$-tuple $ (D_1, D_2, \dots, D_n) $ from $ \cA
$ and $ \cB $, i.e. each $D_i$ is a polynomial in the
matrices from one of the ensembles; , adjacent elements come
from different ensembles, and $ \E(\Tr( D_l )) =0 $ for all
$ l $; we have that
\[ 
\lim_{  N \rightarrow \infty} \E(\tr ( D_1 D_2 \cdots D_n )) =
0.
\]

(\textit{ii}) We say that $ \cA $ and $ \cB$ are\emph{
  asymptotically complex second order free} if $ \cA $ and
$\cB$ are asymptotically free, $ \cA \cup \cB $ has a second
order limit distribution and if, for all $ m, n \geq 1$ with
$m + n > 2$ and all centered and cyclically alternating
tuples $ (C_1, \ldots, C_m )$, $ ( D_1, \ldots, D_n )$ of
polynomials from $ \cA $ and $ \cB $, i.e. cyclically
adjacent elements come from different ensembles, and $
\E(\Tr( C_j)) = \E(\Tr( D_l )) =0 $ for all $j$ and $l$, we
have that
\begin{eqnarray*}\lefteqn{
\lim_{N \rightarrow \infty} \{ \mathrm{cov}(\Tr( C_1 \cdots
C_m),\Tr( D_1 \cdots D_n ) )} \\ && \mbox{} -
\delta_{m, n }\cdot \sum_{ k
  = 1}^n \prod_{ i =1}^n \E(\tr( C_i \cdot D_{ k-i })) \} = 0
\end{eqnarray*}
where $k-i$ is interpreted modulo $n$.

\medskip
(\textit{iii}) We say that $ \cA $ and $ \cB$ are\emph{
  asymptotically real second order free} if $ \cA $ and $
\cB $ are asymptotically free, $ \cA \cup \cB $ has a second
order limit $t$-distribution and if, for all $ m, n \geq 1$
with $m + n > 2$ and all cyclically alternating tuples $ (
C_1, \ldots, C_m )$, $ ( D_1, \ldots, D_n )$ of polynomials
alternating between $ \cA $ and $ \cB $, i.e. cyclically
adjacent elements come from different ensembles, such that $
\E(\Tr( C_j)) = \E(\Tr( D_l )) =0 $ ( $ 1 \leq j \leq m $, $
1 \leq l \leq n $ ), we have that
\begin{eqnarray}\lefteqn{
\lim_{N \rightarrow \infty} 
\{ \cov(\Tr( C_1 \cdots C_m), 
        \Tr( D_1 \cdots D_n ) )\label{eq:02}} \\ 
&&\mbox{} - 
\delta_{m, n} \cdot 
\sum_{ k =1}^n [ \prod_{ i =1}^n \E(\tr( C_i \cdot
  D_{ k -i } )) + \prod_{ i =1}^n \E(\tr( C_i\cdot D_{ i - k }^t
  )) ] \} = 0 \notag
\end{eqnarray}
where again, $k-i$ and $i-k$ are interpreted modulo $n$. 
\end{definition}
It was shown in \cite{mp}, Proposition 29, that the
relations above are associative for the subalgebras
generated by the random variables.

\begin{remark}\label{rem:eqn_4}
In \cite{ms} the property (\textit{ii}) was called second
order freeness as only  the complex case was considered. To
emphasize the difference from real second order freeness we
shall in this paper call (\textit{ii}) complex second order
freeness.

If we have ensembles $\cA$ and $\cB$ such that for all $ m,
n \geq 1$ with $m + n > 2$ and all cyclically alternating
tuples $ ( C_1, \ldots, C_m )$, $ ( D_1, \ldots, D_n )$ of
polynomials from $ \cA $ and $ \cB $, i.e. cyclically
adjacent elements come from different ensembles, such that $
\E(\Tr( C_j)) = \E(\Tr( D_l )) =0 $ ( $ 1 \leq j \leq m $, $
1 \leq l \leq n $ ), we have that equation(\ref{eq:02})
holds we shall say that $\cA$ and $\cB$ \textit{satisfy
  equation} (\ref{eq:02}).
\end{remark}

\begin{proposition} \label{prop:ord}
Let $ \{ A_{1, N}, A_{2, N}, \dots, A_{ m, N },   B_{1, N}, \dots, B_{n, N} \}_N $ be an ensemble of 
 random matrices with the bounded cumulants property
and such that $ \E( \Tr ( B_ l ) ) = 0 $ for all $
1\leq l \leq n $. Then:
\begin{enumerate}

\item[(\textit{i})]
$\E( \Tr( A_1 ) \cdots \Tr( A_m ) ) = \E( \Tr( A_1 )) \cdots
  \E ( \Tr( A_m ) ) + O( N^{ m - 1 } ); $ \text{\emph{in
      particular}}, $ \E( \Tr( A_1 ) \cdots \Tr( A_m ) ) =
  O( N^{ m } ) $.

\item[(\textit{ii})] 
$\E( \Tr( A_1 ) \cdots \Tr( A_m ) \Tr( B_1) \cdots \Tr(B_n)
  ) $

\item[]
\hspace{1cm}$ = \E( \Tr( A_1 ) \cdots \Tr( A_m )) \cdot \E(
\Tr( B_1) \cdots \Tr(B_n) ) + O( N^{ m - 1 } ) $

\item[]
\hspace{7cm}$ = O( N^{m } ).$ \\
In particular, all  cumulants with entries both from 
 $ \{\Tr( A_1 ), \dots,\ab \Tr( A_m ) \} $ 
and from $ \{ \Tr(B_1), \dots, \Tr(B_n) \}$ do vanish asymptotically.
\end{enumerate}
\end{proposition}

\begin{proof}
For (\textit{i}), using the moment-cumulant expansion from
\cite{ns}, we have that
\[
\E( \Tr( A_1 ) \cdots \Tr( A_m ) ) = \sum_{ \pi \in \cP(m) }
\cc_{\pi}( \Tr( A_1 ), \dots, \Tr( A_m ) ).
\]

Since all cumulants of order higher than 2 are bounded, it
follows that $ \cc_{\pi}( \Tr( A_1 ), \dots, \Tr( A_m ) )=
O(N^{ \# \text{of singletons of } \pi } ) $.  But elements
of $ \cP(m)$ have at most $ m $ singletons, with equality
for $ 0_m $, this proves (\textit{i}).

For (\textit{ii}), let $\tau \in \cP(m+n)$ be the partition
with two blocks $(1, \dots, m)$ and $(m+1, \dots,
m+n)$. Note first that the moment-cumulant expansion gives
\begin{eqnarray*}\lefteqn{
\E( \Tr( A_1 ) \cdots \Tr( A_m ) \Tr( B_1) \cdots \Tr(B_n) )} \\
&=& 
\sum_{ \pi \in P( m+n ) } \cc_\pi( \Tr( A_1 ), \dots,
  \Tr(B_n) ) \\
  & = &
\mathop{\sum_{ \pi \in P( m+n ) }}_{\pi < \tau}
\cc_\pi( \Tr( A_1 ), \dots,
  \Tr(B_n) ) \\   
&& \mbox{} + 
\mathop{\sum_{ \pi \in P( m+n ) }}_{\pi \vee \tau = 1_{m+n}}
\cc_\pi( \Tr( A_1 ), \dots,
  \Tr(B_n) ) \\  
  & = &
\E(\Tr(A_1) \cdots \Tr(A_m)) \E(\Tr(B_1) \cdots \Tr(B_n)) \\
&& \mbox{} +  
\mathop{\sum_{ \pi \in P( m+n ) }}_{\pi \vee \tau = 1_{m+n}}
\cc_\pi( \Tr( A_1 ), \dots,
  \Tr(B_n) ). \\ 
\end{eqnarray*}
All cumulants of order higher than 1 are bounded; moreover,
for all indices $ l $, we have that $ k_1( \Tr( B_l ) ) =
0$, since the $ B_l $'s are assumed to be centered.
Let $s$ be the number if singletons of $\pi|_{[m]}$. Then
$s \leq m$ and so
\[
 \cc_\pi( \Tr( A_1 ), \dots, \Tr(B_n))= O( N^s).
\]

But $s = m $ only when $ \pi|_{[m]} = 0_m = \{ (1), \dots,
(m ) \}$, which cannot happen when $\pi \vee \tau =
1_{m+n}$. Thus
\[
\mathop{\sum_{ \pi \in P( m+n ) }}_{\pi \vee \tau = 1_{m+n}}
\cc_\pi( \Tr( A_1 ), \dots,
  \Tr(B_n) ) = O(N^{m-1}).
\] 
This proves (\textit{ii}).
\end{proof}

\begin{remark}\label{rem:ord}
Suppose that $ \sigma \in S_m $ has the cycle decomposition
$ \sigma= c_1\cdot c_2 \cdots c_r $ with each $ c_k = (
i_{k, 1}, i_{ k, 2 }, \dots, i_{ k, s( k) } ) $. For $ A_j
$'s and $ B_l $'s as in Proposition \ref{prop:ord}, denote
\[
\Tr_\sigma (A_1, \dots, A_m) = \prod_{ k=1}^r \Tr( A_{ i_{
    k, 1} } \cdots A_{ i_{ k, s( k ) } } ).
\]
With this notation, Proposition \ref{prop:ord} implies that
\[
 \E( \Tr_\sigma ( A_1 A_2 \cdots A_m ) ) = O( N^r ),
\]
respectively
\begin{align*}
 \E( \Tr_\sigma ( A_1 A_2 \cdots A_m )& \Tr(B_1) \cdots
 \Tr(B_n) ) \\ 
& = 
\E( \Tr_\sigma ( A_1 A_2 \cdots A_m
 ))\cdot \E ( \Tr(B_1) \cdots \Tr(B_n) ).
\end{align*}

\end{remark}

\subsection{Combinatorial results}\label{s22}

For $ n $ a positive integer, let, as in Section
\ref{notations}, $ [ n ] =\{ 1, 2, \dots, n \}$: denote $[ -
  n ] = \{ - n, - n +1, \dots, -1 \} $ and $ [ \pm n ] = [ -
  n ] \cup [ n ]$.  As in \cite{mp}, we will regard a
permutation $ \sigma $ on $ [ n ]$ also as a permutation on
$[ \pm n ] $ where for $ 1 \leq k \leq n $ we let $ \sigma (
- k ) = - k $.

The sets of permutations on $ [ n ] $, respectively $ [ \pm
  n ] $ such that each one of their cycles has exactly 2
elements will be denoted by $ \pn$, respectively $ \pnn $ (
if $ n $ is odd, then by convention $ \pn = \emptyset $).
It was shown in \cite{mp} that if $ p, q \in \pnn $, then $
pq $ has a cycle decomposition $c_1 c_1' \dots c_l, c_l'$
with the property that for all $ k $, if $ c_k = ( i_1, i_2,
\dots, i_l )$, then $ c_k^\prime = ( q( i_l ), q ( i_{ l-1
}), \dots, q( i_1 ) ) $.

We will denote by $ \delta $ the element of $ \pnn $ given
by $ \delta( k ) = - k $.  If $ \sigma $ is a permutation on
some finite set, the number of cycles of $ \sigma $ will be
denoted by $ \#( \sigma ) $.

For a multi-index $\bi\in [ N ]^{ [ \pm n ] }$, and $ p \in
P_2 ( \pm n) $ we write $ \bi =\bi \circ p $ to mean that
for every $ (r, s) \in p$ we have $ i(r) = i(s) $.  Consider
now $ A = (a_{ i j } )_{ i, j =1 }^N $ to be a $ N \times N
$ random matrix; we will use the notation from Section
\ref{notations}, namely $A^{(1)} = A$ and $A^{(-1)} = A^t$.
For $\epsilon = (\epsilon_1, \epsilon_2, \dots, \epsilon_n)
\in \{ -1, 1 \}^n$ and $\pi \in S_n$, let $\Tr_{(\pi,
  \epsilon)}(A_1, \dots, A_n) = \Tr_\pi(A^{(\epsilon_1)}_1,
\dots, A^{(\epsilon_n)}_n)$

We saw that the cycle decomposition of $p\delta$ may be
written $c_1 {c_1}' \cdots\ab c_s {c_s}'$ where ${c_i}' =
\delta c_i^{-1} \delta$.  Since $c'' = c$, it is arbitrary
which of the pair $\{ c_i, {c_i}'\}$ is called $c_i$ and
which ${c_i}'$.

For each $i$, choose a representative of each pair $\{ c_i,
{c_i}'\}$, say $c_1, c_2, \dots,\ab c_s$.  We shall
construct a permutation $\pi( p ) \in S_n$ and $\epsilon \in
\{ -1, 1 \}^n$ which will depend on our choice of $\{c_1,
\dots, c_s\}$.  For each $i$ we construct a cycle $\tilde
c_i$ as follows. Suppose $c_i = ( l_1, \dots, l_r)$.  Let
$\tilde c_i = (j_1, j_2, \dots, j_r)$ where $ j_k = | l_k |
$ and $ \epsilon_{j_k} = { l_k }/{ | l_k | } $.
Define $\pi (p ) = \tilde c_1 \cdots \tilde c_s$ and
$\epsilon = (\epsilon_1, \dots, \epsilon_n)$.

The following result was proved in \cite{mp}, Lemma 5.
\begin{lemma}\label{lemma:pi_epsilon} 
With the notations above, we have that
\[
\mathop{\sum_{1_{\pm 1}, \dots, i_{\pm n}=1}}_{i = i \circ
  p}^N a^{(1)}_{i_1 i_{ -1 }} a^{(2)}_{i_2 i_{ - 2 } }
\cdots \cdots a^{(n)}_{i_{ n}i_{ - n}} = \Tr_{ \pi ( p ) }
(A^{(\epsilon_1)}_1, \dots, A^{(\epsilon_n)}_n).
\]
\end{lemma}

\begin{remark}\label{rem:2}
In the framework above, if $ p([ n ]) = [ - n ] $, then $ p
\delta ( [ n ] ) = [ n ] $, so the cycles of $ p \delta $
have either only positive or only negative elements,
therefore $ \epsilon_j = 1 $ for all $ j = 1, 2, \dots, n $.
\end{remark}

\subsection{ Results on asymptotic second order freeness}

In this section we will recall some results from
\cite{mp} and \cite{cs}.

\begin{theorem}\label{thm:8}
Let $ \cA= \{ A_{ 1 }, A_{ 2 }, \dots, A_{ s } \}_{ N \geq 1
} $ and $ \cB = \{ B_{ 1 }, B_{ 2 }, \dots,\ab B_{ p } \}_{
  N \geq 1 } $ be two ensembles of $N \times N$ random
matrices such that the entries of $ \cA $ and the entries of
$ \cB $ form two independent families of random
variables. Let $ \mathcal{O} = \{ O \}_{ N \geq 1 } $ be a
family of $ N \times N $ Haar distributed orthogonal random
matrices with entries independent from the entries of $ \cA
$ and $ \cB $.

\emph{(\textit{i})} If the ensemble $ \cA $ has the $ t $-bounded
cumulants property, then $ \cA $ and $
\mathcal{O} $ satisfy equation \emph{(\ref{eq:02})} $($as in
Remark \ref{rem:eqn_4}$)$ and the ensemble $ \mathcal{ O }
\cup \cA $ has the $ t $-bounded cumulants property.
If $ \cA $ has a second order limit $ t $-distribution, then
$ \cA $ and $ \mathcal{O} $ asymptotically real second order
free.

\emph{(\textit{ii})} If the ensembles $ \cA $ and $ \cB $
both have the $t$-bounded cumulants property, then the
ensembles $ \widetilde{\cA} = \{ O A O^t : A\in \cA\}$ and $
\cB $ satisfy equation \emph{(\ref{eq:02})}, and $
\widetilde{\cA } \cup \cB $ has the  $ t $ -bounded cumulants property.
 If $ \cA $ and $ \cB $ have second order
limit $ t$-distributions, then $ \widetilde{\cA} $ and $ \cB
$ are asymptotically real second order free.

\end{theorem}

\begin{proof}
We begin with the proof of (\textit{i}). In Remark 39 of
\cite {mp} we gave two conditions sufficient for the
validity of equation (\ref{eq:02}). These were the
conditions that $\cA$ have the $ t$ -bounded cumulants property.
  If $\cA$ has a second order limit $ t $-distribution,
then the claim that $\cA$ and $\cO$ are
asymptotically free was proved in Theorem 50 of \cite{mp}.

Next we prove (\textit{ii}). We must show that $\widetilde
\cA \cup \cB$ has  the bounded cumulants  property. Let
us recall the formula
\begin{eqnarray}\label{eq:orthogonal_expansion}\lefteqn{
\E(\Tr_\gamma( O^{\epsilon_1} Y_1, \dots , 
    O^{\epsilon_n}Y_n))} \\ \notag
& = &
\sum_{p, q \in \cP_2(n)}
\langle \Wg(p), q \rangle
\E(\Tr_{\pi_{p \cdot_\epsilon q}, \eta_{p \cdot_\epsilon q}}
(Y_1, \dots, Y_n)
\end{eqnarray}
from \cite[Proposition 12]{mp}. Note that here $\Wg$ denotes
the orthogonal Weingarten function. When $\epsilon$ has the
special form $\epsilon_k = (-1)^{k+1}$ we have by
\cite[Proposition 13]{mp} that $\pi_{p \cdot_\epsilon q}$ is
parity preserving, i.e. all cycles consist of numbers with
the same parity. This means that on the right hand side of
(\ref{eq:orthogonal_expansion}) only words in elements all
from $\cA$ or all from $\cB$ appear. Thus as far as the
asymptotics of (\ref{eq:orthogonal_expansion}) are concerned
we may assume that $\cA \cup \cB$ has the $ t $-bounded cumulants property.
 By (\textit{i}) we then have that
$\widetilde \cA \cup \cB$ satisfies equation
(\ref{eq:02}). Finally if $\cA$ and $\cB$ have second order
limit distributions then as they satisfy equation
(\ref{eq:02}) there are asymptotically real second order
free.
\end{proof}

\noindent 
We will need to make use of the unitary Weingarten function
$\Wg$. For positive integers $N$ and $n$ we have a function
$\Wg_{N}: S_n \rightarrow \bC$ which can be used to evaluate
integrals with respect to Haar measure, $d\cU$, on the group
$\cU(N)$ of $N \times N$ unitary matrices. We will use the
notation $ u_{i,j} $ for the function $ u_{i,j}: \cU ( N )
\longrightarrow \mathbb{C} $ that gives the value of the
$(i,j)$-entry. We also need to use the M{\"o}bius function
$\mu$ of the partially ordered set of non-crossing
partitions and apply it to a permutation. Recall that $\mu:
NC(n) \times NC(n) \rightarrow \bR$ is multiplicative and
satisfies $\mu(0_n, 1_n) = (-1)^{n-1} C_{n-1}$ where $C_n =
\binom{2n}{n}/(n+1)$ is the $n^{th}$ Catalan number (see
\cite[Thm. 10.15]{ns}). Let $\sigma \in S_n$ be a
permutation with cycle decomposition $c_1 c_2 \cdots
c_k$. If we consider $\sigma$ to be a partition whose blocks
are the cycles $c_1, \dots, c_k$ then we set $\moeb(\sigma)
= \mu(0_n, \sigma)$. Corollaries 2.4 and 2.7 from \cite{cs}
can be formulated as follows.

\begin{theorem}\label{thm:unit1}
For all positive integers $ n, m$ and all multi-indices $
\bi, \bi^\prime \in [N]^{[n]} $, respectively $ \bj,
\bj^\prime \in [ N ]^{ [m] } $ we have that
\begin{align*}
\int_{\cU ( N ) }
  u_{i_1,j_1}&\cdots u_{i_n, j_n}
 \overline{u_{i^\prime_1, j^\prime_1}}\cdots 
\overline{u_{i^\prime_n, j^\prime_n}} 
\, d\mathcal{U} 
=\\
& \hspace{2.5cm}
\delta_{ m, n } \cdot \sum_{\sigma, \tau\in S_n}
\delta^{\bi}_{ \bi^\prime \circ \sigma} 
\delta^{ \bj }_{ \bj^\prime \circ \tau }
 \cdot \Wg_N(\tau\sigma^{-1}).
\end{align*}
Moreover
\[
\Wg_N( \sigma ) = N^{ -2 n + \#(\sigma)  } \moeb( \sigma ) + 
O( N^{-2 n + \#(\sigma)-2 } ).
\]
\end{theorem}

\section{Haar Distributed Random Unitary Matrices }

We will use a slight reformulation of Theorem
\ref{thm:unit1} in the spirit of \cite{mp}.
First we define the map $ S_n \ni \sigma \mapsto \hat{
\sigma } \in \cP_2 ( 2 n) $ given by $ \hat{\sigma}( i ) =
2n + 1 - \sigma ( i ) $ for $ 1 \leq i \leq n $ .  Note that
$ \hat{\sigma}( 2n + 1 - i ) = \sigma^{-1}( i ) $ and that $
\sigma \mapsto \hat{\sigma} $ is a bijection from $ S_n $ to
the set $ \{ \pi \in \cP_2 ( 2n ) : \pi( [ n ] ) = [ 2n ]
\setminus [ n ] \}$. See Fig. \ref{fig-4}.

\begin{lemma}\label{prop:u1}
Let $ \sigma, \tau $ be permutations from $ S_n $ and $c_1
c_2 \cdots c_p$ be the cycle decomposition of $ \tau\sigma^{
  -1} $.  Then $ \hat{ \tau }\hat{ \sigma } $ has the cycle
decomposition $d_1 d_1' d_2 d_2' \cdots d_p d_p'$ where, if
$ c_k = ( i_1, i_2, \dots, i_l ) $ then $ d_k = ( 2n + 1 -
i_1, 2 n + 1 - i_2, \dots, 2 n + 1 - i_l ) $ and $
d_k^\prime = ( \sigma^{-1} ( i_l ) , \sigma^{-1} ( i_{ l - 1
}), \dots, \sigma^{-1} ( i_1 ) ) $ .
\end{lemma}

\begin{proof}
Suppose $ ( i_1, \dots, i_l ) $ is a cycle of $ \tau
\sigma^{ -1} $ and denote $ j_k = \sigma^{-1} ( i_k ) $, $ 1
\leq k \leq l $.  To simplify the writing, we will consider
that $ i_{l + 1 } = i_1 $ and $ j_{l + 1 } = j_1 $.  Then
$
\hat{ \tau }\hat{ \sigma } ( 2 n + 1 - i_k ) = \hat{ \tau }
( \sigma^{ -1} ( i _k ) ) = \hat{\tau}( j_k ) = 2 n + 1 -
i_{ k + 1}
$,
while
\[
\hat{ \tau } \hat{ \sigma }( \sigma^{-1} ( i_k) ) =
\hat{\tau}\hat{\sigma}\hat{\sigma}( 2 n + 1 - i_k ) =
\hat{\tau}( 2 n + 1 - i_k ) = j_{k-1}=\sigma^{-1}(i_{k-1} ).
\]
\end{proof}

\setbox1=\hbox{\includegraphics{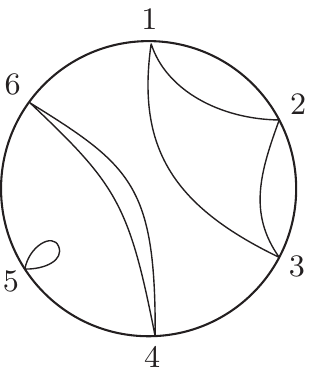}}
\setbox2=\hbox{\includegraphics{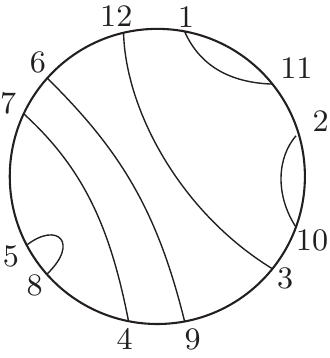}}

\begin{figure}
\noindent
\hfill\box1\hfill\box2\hfill\hbox{}
\caption{\small On the left we have a permutation $\sigma
 \in S_6$ with cycle decomposition $(1,2,3)(4,6)(5)$. On
  the right we have the corresponding pairing $\hat\sigma
  \in \cP_2(12)$ with pairs $(1,11)\ab(2, 10) (3,12) (4,
  7)(5, 8)(6, 9)$. Under $\hat\sigma$ each number in $[6]$
  is paired with a number in $[7, 12]$.}\label{fig-4}
\end{figure}

We introduce now the map $ \Phi_N : \cP_2(n) \times \cP_2(n)
\longrightarrow \mathbb{C} $ as follows. Let $ p, q \in
\cP_2(n) $ be pairings and write the cycle decomposition of
$ pq $ as $c_1 c_1'\cdots c_l c_l'$ with $ c_k^\prime =
qc_k^{-1}q$. Let $m = n/2$, the cycles $c_1 \cdots c_l$ form
a permutation of subset of $[n]$ with $m$ elements. Choose
$\sigma \in S_m$ so that $\sigma$ has the same cycle
structure as $c_1 \cdots c_l$. By this we mean that $\sigma$
has $l$ cycles and the lengths of the cycles of $\sigma$
match those of $c_1 \cdots c_l$. Let $\Phi_N(p,q) =
\Wg_N(\sigma)$. As shown in \cite{cs}, the value $
\text{Wg}_{N } ( \sigma ) $ depends only on the cycle
decomposition of $ \sigma $, so $ \Phi_N $ is well-defined.
Note that Theorem \ref{thm:unit1} then gives that
\begin{equation}\label{eq1}
\Phi_N ( p, q ) = N^{ - n + \#( pq )/2 } \cdot \moeb (
\sigma) + O( N^{ - n + \#( pq)/2 - 2 } ).
\end{equation}
With the above notation, Theorem \ref{thm:unit1} can be
reformulated using pairings. Let $\alpha_1, \dots , \alpha_n
\in \{-1, 1\}$. We shall, as usual, regard $\alpha$ as the
$n$-tuple $(\alpha_1, \dots, \alpha_n)$ and the function
$\alpha: [n] \rightarrow \{-1, 1\}$. Let
\[
\cP_2^\alpha ( n ) = \{ p \in \cP_2 ( n ) \mid \text{ if }
p( k ) = l \text{ then } \alpha_k = -\alpha_l \}.
\]
In the language of functions we can write the property as
$\alpha \circ p = - \alpha$. Theorem \ref{thm:unit1} has the
following reformulation which extends the version stated in
\cite{cm}. For the reader's convenience we provide a
proof. Here we are following our convention that
$u_{ij}^{(1)} = u_{ij}$ and $u_{ij}^{(-1)} = \overline{
  u_{ij}}$. Recall our convention regarding permutations
from \ref{s22}, namely that if a $\pi$ is a permutation of
$[n]$, we extend $\pi$ to be a permutation of $[\pm n]$ by
setting $\pi(-k) = -k$ form $k > 0$.
\begin{proposition}\label{prop:u2}
Then
\begin{equation}\label{eq:pairingweingarten}
\int_{ \mathcal{U}(N)} u_{ i_1, i_{ -1}}^{ (\alpha_1) }
\cdots u_{ i_n, i_{-n}}^{ (\alpha_n) } d \mathcal{U} =
\sum_{ p, q \in P_2^{\boldsymbol{\alpha}}( n ) } \delta_{
  \bi \circ p\delta q \delta} ^{ { \bi } } \cdot \Phi_N( p,
q ).
\end{equation}
\end{proposition}

\begin{proof}
If the cardinalities of $\alpha^{-1}(1)$ and
$\alpha^{-1}(-1)$ are different that both sides of
(\ref{eq:pairingweingarten}) are 0. That the left hand side
is 0 follows from Theorem \ref{thm:unit1} and that the right
hand side is 0 is because in this case $\cP^\alpha_2(n)$ is
empty. So we shall assume that $n = 2m$ and both of
$\alpha^{-1}(1)$ and $\alpha^{-1}(-1)$ have cardinality $m$.

Let $ A= \alpha^{-1}(-1) \cap [m]$ and $B = \alpha^{-1}(1)
\cap [m+1, n]$. Then $A$ and $B$ must have the same
cardinality, $k$. Write $ A = \{ s_1, s_2, \dots, s_k \} $
and $ B = \{ s_{ - 1 } , s_{ - 2 }, \dots, s_{ - k} \} $.
Define $ \pi \in S_{ 2 m } $ as follows
\[ 
\pi(k) = \left\{
\begin{array}{ cl}
k & \text{if}\ k \in [ 2m] \setminus (A \cup B )\\
s_{ - p } & \text{if} \ k = s_p \in A \cup B.
\end{array}
\right.
\]
Let $ \widetilde{\alpha} = \alpha \circ \pi$,  then
\[
 \widetilde{\alpha}_k = \alpha\circ \pi (k) =
\left\{
\begin{array}{c c }
1 & \text{if}\ k \leq m \\
-1 & \text{if}\ k > m.
\end{array}
\right.
\]
Note that $ \pi $ is an involution and for $p \in P^{
  \balpha }_2 (2m ) $, we have $\tilde{p} : = \pi \circ p
\circ \pi \in \cP^{ \widetilde{\alpha}}_2 (2 m ) $. Moreover
since $\tilde p \tilde q = \pi pq \pi$, $\tilde p \tilde q$
and $pq$ have the same cycle structure and so $\Phi_N ( p,
q) = \Phi_N ( \widetilde{p}, \widetilde{q} ) $.  Also $i = i
\circ p $ is equivalent to $ i \circ \pi = ( i \circ \pi )
\circ \tilde{p } .$
Therefore, without loss of generality, we can suppose that $
\alpha_1 = \cdots = \alpha_m = 1$ and $ \alpha_{ m +1 } =
\cdots = \alpha_{ 2m } = -1$.  With this assumption, the map
$ \sigma \mapsto \widehat { \sigma } $ is a bijection from $
S_m $ to $ \cP_2^{\alpha } (2m ) $.

For $ 1 \leq k \leq m $ let $ i_k^\prime = i_{2m+1 - k } $,
$ j_k = i_{ -k}$ and $ j_k^\prime = i_{-(2m + 1 - k)}$. Then
\[
\int_{\cU(N)} u_{i_1i_{-1}}^{(\alpha_1)} \cdots
u_{i_ni_{-n}}^{(\alpha_n)} \, d\cU = \int_{\cU(N)}
u_{i_1j_1} \cdots u_{i_mj_m} \overline{u_{i_1'j_1'}} \cdots
\overline{u_{i_m'j_m'}} \, d\cU.
\]
Also the equality $i = i' \circ \sigma$ means that for $k >
0$, $i(k) = i(\hat \sigma(k))$. Likewise the equality $j =
j' \circ \tau$ means that for $k > 0$, $i(-k) = j(k) = i
(-\hat\tau(k))$. Together they are equivalent to $i = i
\circ \hat\sigma \delta\hat\tau\delta$.  Theorem
\ref{thm:unit1} gives
\begin{eqnarray*}\lefteqn{
\int_{\cU(N)} 
u_{ i_1, i_{ -1}}^{(\alpha_1)}  \cdots 
u_{ i_n, i_{-n}}^{ \alpha_n } d\mathcal{U} } \\
& = &
\sum_{ \sigma, \tau \in S_m }  
\delta^i_{i'\circ\sigma} 
\delta^j_{ j'\circ\tau} \Wg_N(\tau\sigma^{-1})\\
& = &
\sum_{ \sigma, \tau \in S_{m} } 
\delta^i_{ i \circ \hat\sigma \delta \hat\tau \delta}
\Phi_N ( \hat{ \sigma },  \hat{ \tau } )
=
\sum_ { p,q \in P_2^\alpha( 2m ) } \delta_{ \bi \circ p
  \delta q \delta }^{{ \bi } } \cdot \Phi_N ( p , q ).
\end{eqnarray*}
\end{proof}

\medskip 
For the rest of this section, $ U_N $ will denote a $ N
\times N $ Haar distributed unitary random matrix, and $
\mathfrak{U} $ will denote the ensemble $ \{ U_N, U_N^t,
U_N^\ast, \ab \overline{U_N} \}_N $, i.e., with the
notations from Section \ref{notations}, $ \mathfrak{U} = \{
U_N^{ ( \epsilon, \eta ) } : ( \epsilon, \eta) \in \{ -1, 1
\}^2 \} $.  Here for a matrix $A = (a_{ij})_{ij}$,
$\overline{A} = ({A^*})^t =(\overline{a_{ij}})$ denotes the
matrix obtained by taking the complex conjugate of each
entry. As before, we will omit the index $ N $ under the
assumption that we will only multiply matrices of the same
size.

\begin{definition}
An ensemble $ \cA = \{A_1, \dots, A_s\}$ of random matrices
is said to be \emph{unitarily invariant} if for any positive
integer $ N $, any $ N \times N $ unitary matrix $ U$, any
$r_1, \dots, r_k$, and any $i_1, \dots, i_n, i_{-1}, \dots,
i_{-n}$, we have that
\[
 \E( a^{ ( r_1 ) }_{ i_1, i_{ -1 } } a^{ (r_2 ) }_{ i_2, i_{ - 2
 } } \cdots a^{ ( r_k ) }_{ i_s, i_{ - s } } ) = \E( b^{ ( r_1 )
 }_{ i_1, i_{ -1 } } b^{ (r_2 ) }_{ i_2, i_{ - 2 } } \cdots
 b^{ ( r_k ) }_{ i_s, i_{ - s } } )
\]
where $ a^{( k )}_{ i j } $ is the ($i,j $)-entry of $ A_k $
and $b^{ ( k ) }_{i j } $ is the ($ i, j $)-entry of $ U A_k
U^{-1}$.

If this holds for $U$ replaced by any orthogonal matrix we
say the ensemble $\cA$ is \textit{orthogonally invariant}.
\end{definition}

Note that by the invariance of Haar measure for any unitary
matrix $V$ and any Haar distributed random unitary matrix
$U$, we have that $VUV^{-1}$ is another Haar distributed
random unitary matrix. Since any orthogonal matrix is a
unitary matrix we get the same conclusion if we replace $V$
by an orthogonal matrix. Now consider $U$ and $U^*$
together. Let $A_1 = U$, $A_2 = U^*$, $B_1 = VA_1V^{-1}$,
and $B_2 = VA_2V^{-1}$. Then by the unitary invariance of
the ensemble $\{ U, U^*\}$ we have
$\E(a^{(1)}_{11}a^{(2)}_{22}) =
\E(b^{(1)}_{11}b^{(2)}_{22})$ for any unitary $V$. In the
lemma below we shall show that when we set $A_2 =
\overline{U}$ and for a particular $V$, we have
$\E(a^{(1)}_{11} a^{(2)}_{22}) \not=
\E(b^{(1)}_{11}b^{(2)}_{22})$, thus demonstrating that the
ensemble $\{U, \overline{U} \}$ is not unitarily invariant.

\begin{notation}
Let $ \{ U_N \}_N $ be an ensemble of Haar unitary random
matrices, we shall denote by $ \mathfrak{U} $ the ensemble
$\{ U_N, U_N^t, U_N^\ast, \overline{U} \}_N $.
\end{notation}

\begin{lemma}\label{orthoginvar}
The ensemble $ \mathfrak{U} $ is orthogonally invariant, but
not unitarily invariant.
\end{lemma}

\begin{proof}
That $\fU$ is orthogonally invariant follows from the remark
above and the identity
\[
 ( O U O^t )^{( \varepsilon, \eta )} = O U^{( \varepsilon,
  \eta )} O^t
\]
for all $ ( \varepsilon, \eta ) \in \{ -1, 1 \}^2 $ and $ O
$ an orthogonal matrix.

Indeed, let us denote by $ E_{ i, j } $ the $ N \times N $ matrix
with all entries $ 0 $ except for the $ ( i, j ) $-entry
which is 1.  Also, for $ A $ a random matrix, denote $ ( A
)_{ i, j } $ its $ (i, j ) $ entry.

For a positive integer $ r $, a multi-index $ \bi \in [ N
]^{ [ \pm s ] } $ and $\boldsymbol\varepsilon,
\boldsymbol\eta \in \{ -1, 1 \}^s$, define the function $
f_{ \bi, \boldsymbol\varepsilon, \boldsymbol\eta}:
\mathcal{U}(N) \lra \CC $ given by
\[
f_{ \bi, \boldsymbol\varepsilon, \boldsymbol\eta} ( U ) =
\prod_{ k=1 }^s ( U^{ ( \varepsilon_k, \eta_k )} )_{ i_k i_{
    - k } } = \Tr ( \prod_{ k=1 }^s E_{ 1 i_k } U^{ (
  \varepsilon_k, \eta_k )} E_{ i_{ -k} 1 } ).
\]

According to Theorem \ref{thm:unit1}, $ f_{ \bi,
  \boldsymbol\varepsilon, \boldsymbol\eta} \in L^\infty(
\mathcal{U}(N ), d\,\mathcal{ U } ) $ for all $ \bi,
\boldsymbol\varepsilon, \boldsymbol\eta $ as above, and the
translation invariance of Haar measure gives that
\[
 \int_{ \mathcal{U}(N) }
  f_{ \bi, \boldsymbol\varepsilon, \boldsymbol\eta} (U) d\,\mathcal{U}
   = \int_{ \mathcal{U}(N) }
    f_{ \bi, \boldsymbol\varepsilon, \boldsymbol\eta} ( O U O^t ) d\,\mathcal{U}.
\]
Hence
\begin{align*}
\E \Big(\prod_{ k=1 }^s  
( U^{ ( \varepsilon_k, \eta_k )} )_{ i_k i_{ - k } }  \Big)
=&
\int_{ \mathcal{U}(N) } f_{ \bi, \boldsymbol\varepsilon,
  \boldsymbol\eta} (U) d\,\mathcal{U}
=
\int_{ \mathcal{U}(N) }
 f_{ \bi, \boldsymbol\varepsilon, \boldsymbol\eta} ( O U O^t )
  d\,\mathcal{U}\\
=& 
\E \Big(
 \Tr \Big(
  \prod_{ k=1 }^s E_{ 1 i_k } ( O U O^t )^{ ( \varepsilon_k,
    \eta_k )} E_{ i_{ - k } 1 }
   \Big)
 \Big)
\\
=&
 \E \Bigg(
  \Tr \Big(
   \prod_{ k=1 }^s \E_{ 1 i_k } ( O U^{ ( \varepsilon_k, \eta_k )} O^t )
    \E_{ i_{ - k } 1 }
    \Big)
  \Big)
\\
=& 
\E \Big(
 \prod_{ k=1 }^s ( O U^{ ( \varepsilon_k, \eta_k )} O^t )_{i_k i_{-k } }
  \Big).
\end{align*}

For the second part, let $ V $ be the $N \times N$ unitary matrix
\[
V = 
\left(\begin{matrix}
\sin \theta & i\cos \theta & 0 & \cdots & 0 \\
i\cos \theta & \sin \theta & 0 & \cdots & 0 \\
0 & 0    & 1 & \cdots & 0 \\
\vdots & \vdots & \vdots  &  & \vdots \\
0 & 0 & 0 & \cdots & 1 
\end{matrix}\right)
\]
where $\theta$ is such that $\cos\theta \sin \theta \not =
0$. Let $A_1 = U$ and $A_2 = \overline{U}$ and let $B_1 =
VA_1V^*$ and $B_2 = VA_2V^*$. Then $a^{(1)}_{11} = u_{11}$
and $a^{(2)}_{22} = \overline{u_{22}}$ and
\begin{align*}
b^{(1)}_{11} & = \sin^2 \theta u_{11} + i \cos \theta \sin
\theta u_{21} - i \cos \theta \sin \theta u_{12} + \cos^2
\theta u_{22} \\ b^{(2)}_{22} & = \cos^2 \theta
\overline{u_{11}} - i \cos \theta \sin \theta
\overline{u_{21}} + i \cos \theta \sin \theta
\overline{u_{12}} + \sin^2 \theta \overline{u_{22}}.
\end{align*}
Recall from Proposition \ref{prop:u2} that $\E(u_{ij}u_{kl})
= 0$ for all $i, j, k, l$ and $\E(u_{ij}\overline{u_{kl}})
\not= 0$ only when $i=k$ and $j = l$, in which case
$\E(u_{ij}\overline{u_{ij}}) = N^{-1}$. So $\E(a^{(1)}_{11}
a^{(2)}_{22}) = 0$. On the other hand in $\E(b^{(1)}_{11}
b^{(2)}_{22})$ there are sixteen terms of which only four
are non-zero. So
\begin{eqnarray*}
\E(b^{(1)}_{11} b^{(2)}_{22}) & = & \cos^2 \theta \sin^2
\theta \E(u_{11} \overline{u_{11}} + u_{21}
\overline{u_{21}} + u_{12} \overline{u_{12}} + u_{22}
\overline{u_{22}}) \\ & = & 4 \cos^2 \theta \sin^2 \theta
N^{-1}.
\end{eqnarray*}
Thus $\E(a^{(1)}_{11} a^{(2)}_{22}) \not = \E(b^{(1)}_{11} b^{(2)}_{22})$. 
\end{proof}


\begin{definition}\label{defn:016}
Let $\cA $ be a set of $ N \times N $ random matrices. A
word $ W $ in elements from $\cA $ and $ N \times N $
unitary matrices from $ \mathfrak{U} $ will be called
\emph{simplified} if either $ W $ is a centered polynomial
in matrices from $ \cA $ with complex coefficients, or
\[
W = A_1 U^{( \epsilon_1, \eta_1)} A_2 U^{( \epsilon_2,
  \eta_2 )} \cdots A_m U^{ ( \epsilon_m, \eta_m )}
\]
with each $ \epsilon_i, \eta_i \in \{ -1, 1\} $ and each $
A_i $ either a centered polynomial in matrices from $ \cA $
with complex coefficients, or the identity matrix, but $ A_i
$ can be the identity matrix only if $ (\epsilon_{ i -1},
\eta_{ i-1 } ) \neq ( - \epsilon_i, - \eta_i) $, for $ i > 1
$, respectively only if $ ( \epsilon_1, \eta_1 ) \neq ( -
\epsilon_m, - \eta_m ) $, for $ i = 1 $.
\end{definition}

\begin{lemma}\label{lemma:017}
Let $ \cA $ be a set of $ N \times N $ random matrices and $
W $ be a polynomial with complex coefficients in elements
from $ \cA $ and $ N \times N $ unitaries from $
\mathfrak{U} $. Then there exist $ W_1, \dots, W_m $ some
simplified words in $ \cA $ and $ \mathfrak{U} $ and $c_0,
c_1, \dots, c_m $ some polynomials with complex coefficients
in expectations of normalized traces of products of elements
from $ \cA $ such that
\[ 
\Tr( W ) = c_0 + c_1 \Tr(W_1) + \dots + c_n \Tr(W_m ).
\]
\end{lemma}

\begin{proof}
If $ W $ does not contain elements from $ \mathfrak{U} $,
the result is trivial, since $ W = \centre{W} + \E(\tr(W) )
$, where, for a random matrix $ A $, we take $ \centre{A} =
A - \E(\tr(A)) $.  If $ W $ contains both elements from $
\cA $ and $ \mathfrak{U}$, due to linearity and traciality,
we may suppose that
\[
W = A_1 U^{( \epsilon_1, \eta_1)} A_2 U^{( \epsilon_2,
  \eta_2 )} \cdots A_m U^{ ( \epsilon_m, \eta_m )}
\]
with each $ \epsilon_i, \eta_i \in \{ -1, 1\} $ and each $
A_i $ a polynomial in matrices from $ \cA $ with complex
coefficients.
   
For $ W $ as above, let $ n $ the number of $ A_i $ not as
in Definition \ref{defn:016}. We will show the result by
induction on $ (m, n) $. If $ m = n = 1$, the result is
trivial, since
\[
     \Tr( A_1 U^{( \epsilon_1, \eta_1 )} ) = \Tr(
     \centre{A_1} U^{( \epsilon_1, \eta_1)}).
\]
For the inductive step, using again traciality, let us
suppose that $ A_m $ does not satisfy the conditions from
Definition \ref{defn:016}. Then
     \begin{align*}
      W = A_1 & U^{( \epsilon_1, \eta_1)}  A_2 U^{( \epsilon_2, \eta_2 )} \cdots
        A_{m-1} U^{ ( \epsilon_{ m -1}, \eta_{ m -1 } )}
        \centre{A_m } U^{ ( \epsilon_m, \eta_m )}  \\
         & + 
        \E( \tr(A_m) ) A_1 U^{( \epsilon_1, \eta_1)}  A_2 U^{( \epsilon_2, \eta_2 )} \cdots
                 A_{m-1} U^{ ( \epsilon_{ m -1}, \eta_{ m -1 } )}
                  U^{ ( \epsilon_m, \eta_m )}.
       \end{align*}
The number of $ A_i $'s from the first term of the
right-hand side not satisfying the conditions from
Definition \ref{defn:016} is strictly less than $n $. For
the second term, we distinguish two cases.  If $ ( \epsilon
_{ m -1}, \eta_{ m -1} ) = ( - \epsilon_m, -\eta_m ) $, then
       \begin{align*}
       \Tr(A_1 U^{( \epsilon_1, \eta_1)}  A_2 U^{( \epsilon_2, \eta_2 )}& \cdots
                        A_{m-1} U^{ ( \epsilon_{ m -1}, \eta_{ m -1 } )}
                         U^{ ( \epsilon_m, \eta_m )} )\\
                         &=
        \Tr(A_{ m -1} A_1 U^{( \epsilon_1, \eta_1)}\cdots A_{m-2} U^{ ( \epsilon_{ m -2}, \eta_{ m -2 } )})
                \end{align*}
and the word from the right-hand side has length strictly less than $ 2m $.
If 
 $ ( \epsilon _{ m -1}, \eta_{ m -1} ) \neq ( - \epsilon_m, -\eta_m ) $
then
   \begin{align*}
       A_1 U^{( \epsilon_1, \eta_1)}  & A_2 U^{( \epsilon_2, \eta_2 )} \cdots
                        A_{m-1} U^{ ( \epsilon_{ m -1}, \eta_{ m -1 } )}
                         U^{ ( \epsilon_m, \eta_m )} \\
                         &=
        A_1 U^{( \epsilon_1, \eta_1)}  A_2 U^{( \epsilon_2, \eta_2 )}\cdots
                        A_{m-1} U^{ ( \epsilon_{ m -1}, \eta_{ m -1 } )} I
                         U^{ ( \epsilon_m, \eta_m )}
                \end{align*}
and the number of $ A'i$'s not satisfying the conditions
from Definition \ref{defn:016} in word from the right-hand
side is $ n - 1 $.
\end{proof}

\begin{theorem}\label{thm:distrib}
If $ \cA $ is an ensemble of random matrices which is
independent from $ \mathfrak{U} $ and has the $ t $-bounded
cumulants property, then the ensemble $ \mathfrak{ U } \cup
\cA $ also has the $ t $-bounded cumulants property.

\end{theorem}

\begin{proof}

Following Definition \ref{defn:1}, we must show that for $R
\geq 2$ and any polynomials with complex coefficients $ W_1,
\dots, W_R $ in elements from $ \cA \cup \cA^t $ and $
\mathfrak{U } $, we have that
\begin{enumerate}

\item[\raise0.12em\hbox{$\centerdot$}]
$\sup_N |\E(\tr(W_i))| < \infty$ for each $i$ and

\item[\raise0.12em\hbox{$\centerdot$}]
$\sup_N |k_R(\Tr(W_1), \dots, \Tr(W_R))| < \infty$.  

\end{enumerate}
From Lemma \ref{lemma:017}, we have that 
\[
 \Tr(W_i ) = c_{i, 0} + c_{i, 1} \Tr(W_{i, 1}) + \cdots  + c_{ i, m(i) } \Tr( W_{ i, m(i) } )
\]
where $ c_{i, k }$'s are polynomials in normalized traces of
products of elements from $ \cA $ and $ W_{i, k}$'s are
simplified words in $\cA \cup \cA^t $ and $ \mathfrak{U} $.
Since $ \cA $ has the $ t $-bounded cumulants property, each
$ c_{ i, k } $ is bounded as $ N \lra \infty $, therefore,
using the multilinearity of the cumulants, it suffices to
show the result for $ W_i $ simplified words in $
\cA\cup\cA^t $ and $ \mathfrak{U} $.  Therefore, we may
suppose that
\begin{align}
&W_i  = U^{(\epsilon_{M_{i-1}+1}, \eta_{M_{i-1}+1})}
B_{M_{i-1}+1} \cdots U^{(\epsilon_{M_i}, \eta_{M_i})}
B_{M_i},\label{eq:general-word} \text{ for\ } 1 \leq i \leq r \\
&W_ { r + k } = B_{ M + k }, \  \text{for } 1 \leq k \leq R - r, \nonumber
\end{align}
where $ M_0, M_1, \dots, M_r $ is a strictly increasing
sequence of positive integers with $ M_0 = 0 $, $ M_r = M $,
$ R $ is a positive integer greater than $ r $, $ (
\epsilon_i, \eta_i ) \in \{ - 1, 1 \}^2 $ and $ B_k $ are
centered polynomials in elements from $ \cA $ or identity
matrices, such that
\begin{itemize}
\item[$\centerdot$] for $ 1 \leq i \leq r $, if $ (
  \epsilon_{ M_{ i -1} + 1 }, \eta_{ M_{ i -1} + i } ) = ( -
  \epsilon_{ M_i }, - \eta_{ M_i } ) $, then\\ $ B_{ M_i }
  \neq I $
\item[$\centerdot$] for $ 1 \leq i \leq r $ and $ 0 \leq k
  \leq M_{ i + 1} - M_i -1$, if $ (\epsilon_{M_{i}+k},
  \eta_{M_{i}+k}) = ( - \epsilon_{ M_{i}+k + 1 } , - \eta_{
    M_{i} +k + 1 }) $ then $ B_{ M_i + k } \neq I $
\item[$\centerdot$] $ B_{ M + k } \neq I $,  for $ 1 \leq k \leq R-r $
\end{itemize}

We claim that it suffices to show that for $ W_1, \dots, W_R
$ as above we have
\begin{equation}\label{eq:moment-bound}
\E(\Tr(W_1) \cdots \Tr(W_R)) = O(1).
\end{equation}
Indeed, first note that (\textit{i}) is automatic because 
$\tr = N^{-1}\Tr$. 
For (\textit{ii}) let us recall some notation from [{\sc ns}].

Let $X_1, \dots, X_R$ be random variables and $\pi =
\{V_1, \dots, V_k\} \in \cP(R)$ be a partition of $[R]$. 
We let 
\[\E_\pi(X_1, \dots, X_R) = \prod_{V
  \in \pi}\E\big( \prod_{j\in V} X_j\big). \]
For example,
when $\pi = \{(1,3,5)(2,4)\}$, 
\[
\E_\pi(X_1, \ab \dots, X_5)
= \E(X_1X_3X_5)\ab \E(X_2 X_4).
\]
With this notation we can write the inverse of equation
(\ref{eq:moment-cumulant}) via the M\"obius function, $\mu$,
for the poset $\cP(R)$, see [{\sc ns}, Prop. 10.11 and
  Ex. 10.33]. Namely
\begin{equation}\label{eq:cumulant-moment}
k_R(X_1, \dots, X_R) =
\sum_{\pi \in \cP(R)} \mu(\pi, 1_R) \E_\pi(X_1, \dots, X_R).
\end{equation}
If (\ref{eq:moment-bound}) holds then for all $\pi$ we have that
\[
\E_\pi(\Tr(W_1) \cdots \Tr(W_R)) = O(1).
\]
Hence by equation (\ref{eq:cumulant-moment}) we get
(\textit{ii}).

Thus we must show that for every $\{l_i\}_i, R, \{
(\epsilon_i, \eta_i)\}_i$, $\{b_i\}_i$ satisfying
(\textit{a}), (\textit{b}), and (\textit{c}) we have
equation (\ref{eq:moment-bound}).

Let $(\epsilon_1, \eta_1), \dots, (\epsilon_r, \eta_r) \in
\{-1,1\}^2$.  Denote the $ ( i, j ) $-entry of $ B_k $ by $
b_{ i, j }^{ ( k ) } $. Let us find an expansion
for \[\Tr(U^{(\epsilon_1, \eta_1)} B_1 \cdots
U^{(\epsilon_k, \eta_k)} B_k).\] We shall let $\gamma = (1,
2, 3, \dots, k)$ be the permutation of $[k]$ which has just
one cycle. We have that
\begin{eqnarray*}\lefteqn{
\Tr(U^{(\epsilon_1, \eta_1)} B_1 \cdots U^{(\epsilon_k,
  \eta_k)} B_k)}\\ & = & \mathop{\sum_{i_1, \dots,
      i_k}}_{i_{-1}, \dots, i_{-k}} \big(U^{(\epsilon_1,
    \eta_1)}\big)_{i_1i_{-1}} (B_1)_{i_{-1}i_2} \cdots
  \big(U^{(\epsilon_k, \eta_k)}\big)_{i_ki_{-k}}
  (B_k)_{i_{-k}i_1}.
\end{eqnarray*}

Now $\big(U^{(\epsilon_l, \eta_l)}\big)_{i_li_{-l}} = u_{j_l
  j_{-l}}^{(\eta_l)}$ where $j_l = i_{\epsilon_l l}$. Also
$(B_l)_{i_{-l}i_{\gamma(l)}}= b^{(l)}_{j_{\phi(l)},
  j_{\phi(-l)}}$ where $\phi(l) = -\epsilon_l l$ and
$\phi(-l) = \epsilon_{\gamma(l)} \gamma(l)$, for $1 \leq l
\leq k$.  
It is easy to check that $\phi:[\pm M] \rightarrow [\pm M]$
is a bijection, so we shall regard $\phi$ as a permutation
of $[\pm M]$. Indeed, letting $\epsilon$ be the permutation
of $[\pm M]$ given by $\epsilon(k)= \epsilon_k k$, we may
write $\phi = \epsilon \gamma \delta$ using the notation of
\S\ref{s22}. Thus we may write
\begin{eqnarray*}\lefteqn{
\Tr(U^{(\epsilon_1, \eta_1)} B_1 \cdots U^{(\epsilon_k,
  \eta_k)} B_k)}\\ & = & \mathop{\sum_{j_1, \dots,
      j_k}}_{j_{-1}, \dots, j_{-k}} u^{(\eta_1)}_{j_1j_{-1}}
  b^{(1)}_{j_{\phi(1)}j_{\phi(-1)}} \cdots
  u^{(\eta_k)}_{j_kj_{-k}}
  b^{(k)}_{j_{\phi(k)}j_{\phi(-k)}}.
\end{eqnarray*}
Now let us return to $\E(\Tr(W_1) \cdots \Tr(W_R)) =
\E(\Tr(W_1) \cdots \Tr(W_M) \mfb)$ where $\mfb =
\Tr(B_{M+1})\cdots \Tr(B_{ M+R- r }) $. Now we let $\gamma$
be the permutation with cycle decomposition $(1, \dots,
M_1)(M_1+1, \dots, M_2) \cdots (M_{r-1}+1, \dots, M)$. Then
we have, using the same reasoning as above,
\begin{eqnarray*}\lefteqn{
\E(\Tr(W_1) \cdots \Tr(W_R)) }\\
& = &
\mathop{\sum_{j_1, \dots, j_M}}_{j_{-1}, \dots, j_{-M}} 
\kern-1em\E\big(
u^{(\eta_1)}_{j_1j_{-1}} b^{(1)}_{j_{\phi(1)}j_{\phi(-1)}} \cdots
u^{(\eta_M)}_{j_Mj_{-M}} b^{(M)}_{j_{\phi(M)}j_{\phi(-M)}}\mfb\big) \\
& = &
\mathop{\sum_{j_1, \dots, j_M}}_{j_{-1}, \dots, j_{-M}} 
\kern-1em\E\big(
u^{(\eta_1)}_{j_1j_{-1}} \cdots u^{(\eta_M)}_{j_Mj_{-M}}\big)
\E\big(
b^{(1)}_{j_{\phi(1)}j_{\phi(-1)}} \cdots 
b^{(M)}_{j_{\phi(M)}j_{\phi(-M)}}\mfb\big). 
\end{eqnarray*}
We let $\bj =(j_{-M}, \dots, j_{-1}, j_1, \dots,
j_M)$. Recall that we let $\delta$ be the permutation of
$[\pm M]$ given by $\delta(k) = - k$. Also if $\pi$ is a
permutation of $[M]$ we make $\pi$ a permutation of $[\pm
  M]$ by setting $\pi(-k) = -k$ for $k > 0$. Then for
pairings $p, q \in \cP_2(M)$, $p\delta q \delta$ is a
permutation of $[\pm M]$. By Proposition \ref{prop:u2} we
can write
\[
\E\big( u^{(\eta_1)}_{j_1j_{-1}} \cdots
u^{(\eta_M)}_{j_Mj_{-M}}\big) = \sum_{p,q \in \cP_2^\eta(M)}
\delta ^{\bj}_{\bj\circ p\delta q \delta} \Phi_N(p,q).
\]
Thus
\begin{eqnarray*}\lefteqn{
\E(\Tr(W_1) \cdots \Tr(W_R)) }\\
& = &
\sum_{p,q \in \cP_2^{ \boldsymbol {\eta }  }(M)} \Phi_N(p,q)
\sum_{\bj}
\delta ^{\bj}_{\bj\circ p\delta q \delta}
\E\big(
b^{(1)}_{j_{\phi(1)}j_{\phi(-1)}} \cdots 
b^{(M)}_{j_{\phi(M)}j_{\phi(-M)}}\mfb\big).
\end{eqnarray*}

Let $\tau = \phi^{-1} p \delta q \delta \phi \in
\cP_2(\pm M)$. Then
\[
\sum_{\bj = \bj\circ p\delta q\delta}
b^{(1)}_{j_{\phi(1)}j_{\phi(-1)}} \cdots
b^{(M)}_{j_{\phi(M)}j_{\phi(-M)}} = \sum_{\bi =\bi \circ
  \tau} b^{(1)}_{i_1i_{-1}} \cdots b^{(M)}_{i_Mi_{-M}}.
\]
From Lemma \ref{lemma:pi_epsilon} (see [{\sc mp}, Lemma 5])
there exist a permutation $\pi(p, q) = \pi(\tau)$ from $
S_M$ and $\lambda \in \{-1, 1\}^M$ such that
\[
\sum_{\bi = \bi \circ  \tau} 
b^{(1)}_{i_1i_{-1}} \cdots b^{(M)}_{i_Mi_{-M}}
= \Tr_{\pi(p, q)}(B_1^{(\lambda_1)}, \dots, B_M^{(\lambda_M)}),
\]
with $\pi(p,q)$ constructed as in \S\ref{s22}, (see the discussion
preceding Lemma \ref{lemma:pi_epsilon}). 
Thus
\begin{eqnarray}\label{eq:stepone}\lefteqn{
\E(\Tr(W_1) \cdots \Tr(W_R)) }\\
& = &
\sum_{p,q \in \cP_2^\eta(M)} \Phi_N(p,q)
\E(\Tr_{\pi(p, q)}(B_1^{(\lambda_1)}, \dots, B_M^{(\lambda_M)}) \mfb).
\notag
\end{eqnarray}
Let us note that if $(i)$ is a singleton of $\pi(p, q)$
i.e. a cycle of length 1, then $(i)(-i)$ are cycles of
$\tilde{p}\delta$ where $\tilde{p}= \phi^{-1}p \delta q
\delta \phi$ and thus $(i, -i)$ is a cycle of $\tilde{p}$
and hence $(\phi(i), \phi(-i))$ is a cycle of $p\delta
q\delta$, which in turn implies that $\phi(i)$ and
$\phi(-i)$ have the same sign, i.e. either both are positive
or both negative. Suppose that both $\phi(i) = -\epsilon_i
i$ and $\phi(-i) = \epsilon_{\gamma(i)} \gamma(i)$ are
positive then $\epsilon_i = - \epsilon_{\gamma(i)}$. Also
$p(\phi(i)) = \phi(-i)$ implies that $\eta_i= \eta_{\phi(i)}
= - \eta_{\phi(-i)} = -\eta_{\gamma(i)}$, because $p \in
\cP^\eta(M)$. Thus we have $(\epsilon_i, \eta_i) = -
(\epsilon_{\gamma(i)}, \eta_{\gamma(i)})$. When both
$\phi(i)$ and $\phi(-i)$ are negative we get the same
conclusion: $(\epsilon_i, \eta_i) = - (\epsilon_{\gamma(i)},
\eta_{\gamma(i)})$ (except working with $q$). Thus by
condition (\textit{c}), we have that when $(i)$ is a
singleton of $\pi(p,q)$, we have $\E(\Tr(B_i)) = 0$.

Given $p, q \in \cP_2(M)$, we claim that
\[
\E(\Tr_{\pi(p, q)}(B_1^{(\lambda_1)}, \dots,
B_M^{(\lambda_M)}) \mfb) = O(N^{M/2})
\]
with equality only when $\pi(p,q)$ is a pairing of $[M]$.
Indeed, let $s$ be the number of singletons of $\pi(p,q)$.
By the previous paragraph if (\textit{i}) is a singleton of
$\pi(p,q)$ we have $\E(\Tr(B_i)) = 0$. So by Proposition
\ref{prop:ord}(\textit{ii})
\[
\E(\Tr_{\pi(p, q)}(B_1^{(\lambda_1)}, \dots,
B_M^{(\lambda_M)}) \mfb) = O(N^{t})
\]
where $t = \#(\pi(p,q)) - s$. Thus $t \leq (M - s)/2 \leq
M/2$ with equality only if all cycles of $\pi(p, q)$ have
length 2, i.e. $\pi(p, q)$ is a pairing. Recall from
equation (\ref{eq1}) that $\Phi_N(p,q) =
O(N^{-M+\#(pq)/2})$. So
\begin{equation}\label{eq:end_of_thm_16}
\Phi_N(p,q)\E(\Tr_{\pi(p, q)}(B_1^{(\lambda_1)}, \dots,
B_M^{(\lambda_M)}) \mfb) = O(N^{-M+\#(pq)/2+t}). 
\end{equation}
Combining this with (\ref{eq:stepone}) we get that 
\[
\E(\Tr(W_1) \cdots \Tr(W_R)) = O(1)
\]
as claimed. 
\end{proof}

\begin{remark}\label{remark:mathcalE}
Let us return to equation (\ref{eq:stepone}) to extract a
few observations. Let $ \boldsymbol{A} ( W_1, \dots, W_r)
=\{ p \in \cP_2^{\boldsymbol{\eta}} ( M ) : \pi ( p, p )
\ \text{is a pairing} \}.$ For $ p \in \boldsymbol{A} ( W_1,
\dots, W_r ) $ , we have that $ \Phi_N ( p, p ) \ab=
O(N^{\frac{-M}{2}} ) $. By Proposition
\ref{prop:ord}\,(\textit{ii}) we have for $p \in
\boldsymbol{A} ( W_1, \dots, W_r)$ that
\begin{eqnarray*}\lefteqn{
N^{-M/2}
\E \big( \Tr_{ \pi ( p, p ) }
( B_1^{ ( \lambda_1)} \cdots B_M^{( \lambda_M ) } ) \mfb ) } \\
&  = &
N^{-M/2}  \E ( \Tr_{ \pi ( p, p ) }(
B_1^{ ( \lambda_1)} \cdots B_M^{ ( \lambda_M ) } )\big) \cdot \E
( \mfb ) + O(N^{-M/2-1}).
\end{eqnarray*}
Denoting
\[
\cE ( p, W_1, \dots, W_r)= \Phi_N ( p, p ) \E ( \Tr_{ \pi (
  p, p ) }( B_1^{ ( \lambda_1)} \cdots B_M^{ ( \lambda_M ) }
) ),
\]
the equality above and equations (\ref{eq:stepone}),
and (\ref{eq:end_of_thm_16}) give that
\begin{eqnarray}\label{eq:mathcalE2}\lefteqn{%
\E( \Tr(W_1) \cdots \Tr(W_R) ) }\\ & = & \sum_{ p \in
    \boldsymbol{A} ( W_1, \dots, W_r) } \cE (p, W_1, \dots,
  W_r ) \E (\mfb ) + o(1). \notag
\end{eqnarray}
Note also that $ \cE(p, W_1, \dots, W_r ) = O(1)$ for all $
p \in \boldsymbol{A} ( W_1, \dots, W_r ) $.
\end{remark}
\qed


\begin{corollary}\label{cor:ut}
The ensembles $ \{ U_N, U_N^\ast \}_N $ and $ \{ U_N^t,
\overline{U}_N \}_N $ are asymptotically free.
\end{corollary}

\begin{proof}
By taking $ \cA $ to be the ensemble given by the identity
matrices, property (\textit{i}) in the proof of Theorem
\ref{thm:distrib} implies that
\[
 \E \left( \Tr\left( U^{ p_1} \cdot (U^t)^{ p_2 } \cdots U^{
   p_{ 2M - 1 } } (U^t)^{ p_{ 2 M } } \right) \right) =
 O(1).
\]
Since both $ \{ U_N, U_N^\ast \}_N $ and $ \{ U_N^t,
\overline{U}_N\}_N $ satisfy condition (\textit{i}$^\prime$)
from Definition \ref{defn:1}, it suffices to show that for
any non-zero integers $ p_1, p_2, \dots,\ab p_{2M} $, we
have
\[
\lim_{ N \rightarrow \infty } \E \left(\tr\left( U^{ p_1}
\cdot (U^t)^{ p_2 } \cdots U^{ p_{ 2M - 1 } } (U^t)^{ p_{ 2
    M } } \right) \right) = 0.
\]
but this follows from our first remark, hence the conclusion.
\end{proof}

\begin{remark}\label{remark:free}
For $ U $ a Haar distributed random unitary, $ \E ( \tr (
U^n ) ) = 0 $ for all nonzero integers $ n $, therefore the
definition of free independence and Corollary \ref{cor:ut}
give
\[
\lim_{ N \lra \infty } \E( \tr ( (U^{( \varepsilon, \eta )}
)^m ( U^{ ( \varepsilon^\prime, \eta^\prime )} )^n )) =
\delta_m^n \delta_{ \varepsilon}^{ -\varepsilon^\prime }
\delta_{ \eta }^{-\eta^\prime}.
\]
\end{remark}

\begin{corollary}\label{thm:bdd}
If $ \cA $ is an ensemble of random matrices with the $ t
$-bounded cumulants property and is independent from $
\mathfrak{U} $, then $ \mathfrak{U}$ and $\cA $ satisfy
equation $( \ref{eq:02} ) $ from Definition \ref{defin:03}.
\end{corollary}

\begin{proof}
By Theorem \ref{thm:distrib} ${\cA} \cup \fU$ has a bounded
second order $t$-distribution. By Lemma \ref{orthoginvar},
$\fU$ is orthogonally invariant so we can apply part
(\textit{ii}) of Theorem \ref{thm:8} to conclude that $\fU$
and $\cA$ satisfy equation (\ref{eq:02}).

\end{proof}

\begin{lemma}\label{lemma:4-preliminary_to_thm_22}
Suppose $\cA$ is an ensemble of random matrices independent
from $ \mathfrak{U} $ and with the $ t $-bounded cumulants
property.  Let $ r \geq 3$ and $ m_0, m_1, \dots, m_r $ be a
strictly increasing sequence of positive integers such that
$ m_0 = 0 $ and let $W_1, \dots, W_ r $ be simplified words
in $ \cA $ and $ \mathfrak{U} $ such that each $W_i$ is a
product
\[
W_i = U^{(\epsilon_{m_i+1}, \eta_{m_i+1})} A_{m_i+1} \cdots
U^{(\epsilon_{m_{i+1}}, \eta_{m_{i+1}})} A_{m_{i+1}}
\]
for some $(\epsilon_i, \eta_i) \in \{ 1, -1 \}^2$.
Then
\[
\lim_{ N \lra \infty } k_r (\Tr(W_1), \dots, \Tr(W_r ) )  = 0
\]
\end{lemma}

\begin{proof} 
Denote by $ \cP_{ 1, 2 } (r ) $ the set of all partitions on
$ [ r ] $ whose blocks have no more than 2 elements.  To
prove the lemma it suffices to show that we have
\begin{eqnarray}\lefteqn{\label{eq:011}
\lim_{ N \lra \infty }\E (
\Tr(W_1)\cdots \Tr(W_r )
)}\\
&& \mbox{}- 
\sum_{ \pi \in \cP_{ 1, 2 } ( r ) } \cc_\pi (\Tr(W_1 ),
\dots, \Tr( W_r  ) ) =0, \notag
\end{eqnarray}
as this will inductively force all cumulants higher than 2
to vanish.

By Remark \ref{remark:mathcalE} we have that
\begin{eqnarray*}\lefteqn{
\lim_{N \rightarrow \infty}
\E(\Tr(W_1, \dots, W_r ) ) }\\
&&\mbox{}-
\sum_{p \in \boldsymbol{A} ( W_1, \dots, W_r ) }
\Phi(p,p) \E(\Tr_{\pi(p,p)}(A_1, \dots, A_{m_r}))
= 0
\end{eqnarray*}
where $\boldsymbol{A} ( W_1, \dots, W_r ) =\{ p \in
\cP_2^{\boldsymbol{\eta}} ( m_r ) \mid \pi ( p, p )$ is a
pairing $\}$.

Thus it suffices to show that 
\begin{eqnarray*}\lefteqn{
\lim_{N \rightarrow \infty}
\sum_{ \pi \in \cP_{ 1, 2 } ( l ) } \cc_\pi (\Tr(W_1 ),
\dots, \Tr( W_r ) ) }\\
&& \mbox{}-
\sum_{p \in \boldsymbol{A} ( W_1, \dots, W_r ) }
\Phi(p,p) \E(\Tr_{\pi(p,p)}(A_1, \dots, A_{m_r})) =0.
\end{eqnarray*}

Let $\gamma$ be the permutation of $[ m_r ]$ with the $ r $
cycles
\[
(1, 2, \dots, m_1) (m_1 +1, \dots, m_2) \cdots (m_{ r-1}+1,
\dots, m_r ),
\]
and let $ \overline{\gamma} $ be the partition of $ [ \pm
  m_r ] $ with $ r $ blocks $ B_1, \dots, B_r $ with
\[B_k 
= \ab \{-m_k, \dots, -(m_{k-1}+1)\} \cup \{m_{k-1}+1, \dots,
m_k\}. 
\]
The blocks of the partition $\overline{\gamma}$ are the
union of a cycle of $\gamma$ with the image of the cycle
under the map $\delta$. For $ l,s \in [ r ]$, let $
\overline{\gamma}_{ l, s} = \overline{\gamma}_{ | B_l \cup
  B_s } $.

Let $ p \in \textbf{A} ( W_1, \dots, W_r )$ and suppose that
$ p\delta p \delta \vee \overline{\gamma} $ has the blocks $
C_1, \dots, C_t $.  Each $ C_j $ is a union of blocks of $
\overline{\gamma} $, hence $ p $ determines a partition $
\rho( p ) \in \cP(r ) $, with blocks $ D_1, \dots, D_t $
given by $ l \in D_j $ if and only if $ B_l \subset C_j $.

Let $ \mathcal{W}_j = \{ W_k ; B_k \subseteq C_j \} .$ We
will first show that
\begin{equation}\label{eq:cc22}
\cE(p, W_1, \dots, W_r ) = \cE(p_{ | C_1 },
\mathcal{W}_1)\cdots \cE( p_{| C_t }, \mathcal{W}_t) + o(1).
\end{equation}
To prove (\ref{eq:cc22}) it suffices to prove that the blocks
$ C_j $ are invariant under the permutation $\pi(p, p) $;
because we have the M\"obius function of equation
(\ref{eq1}) is multiplicative in the sense of \cite[Lecture
  1]{ns}. But, as described in Section \ref{s22}, the blocks
of $ \pi(p, p) $ are obtained by taking the absolute value
of the elements of the cycles of $ \delta \widetilde{p} $,
for $ \widetilde{ p } $ defined as in the proof of Theorem
\ref{thm:distrib}, i.e.  $ \widetilde{p} = \phi^{-1} p
\delta p \delta \phi $.  But, by construction, $ \phi( B_l)
= B_l $ (from the definition of $\phi$) and $ \delta ( B_l )
= B_l $ for all $ B_l $; and $ p( C_j ) = C_j $. Therefore $
\delta \widetilde{p } (C_j ) = C_j $ for all $ C_ j $. This
establishes equation (\ref{eq:cc22}).

Next, we shall  show that each $ C_j $ contains at most two
blocks of $ \overline{\gamma} $, which will show
that $\rho( p)\in P_{ 1, 2 } ( r ) $.  In Corollary 22 of
\cite{mp} we showed that if $\pi(p,p)$ (there denoted
$\pi_{p \mskip 2mu\cdot_\epsilon p}$) is a pairing then each
block of $p \vee \gamma$ contains at most two cycles of
$\gamma$. The connection between $p \vee \gamma$ and $
p\delta p \delta \vee \overline{\gamma} $ is quite
simple. With $p \vee \gamma$ we have a partition of $[M]$;
if we reflect this to a partition of $[-M]$ by applying
$\delta$ and then join each block of $p \vee \gamma$ to its
image under $\delta$ we then get $ p\delta p \delta \vee
\overline{\gamma} $. Thus each $C_j$ contains at most two
blocks of $\overline{\gamma}$.

Thus we have 
\[
\sum_{p\in \textbf{A}(W_1, \dots, W_r )} \cE(p, W_1, \dots,
W_r )
=
\sum_{\pi \in \cP_{1,2}(r)} 
\sum_{ 
\substack{ p\in \textbf{A}(W_1, \dots, W_r ) \\ 
\rho(p ) = \pi }
}
 \cE(p, W_1, \dots, W_r )
\]
and thereby have reduced the proof of  to
showing that
\begin{equation}\label{eq:the-final-step}
 \sum_{ 
\substack{ p\in \textbf{A}(W_1, \dots, W_r ) \\ 
\rho(p ) = \pi }
}
 \cE(p, W_1, \dots, W_r )
= \cc_{ \pi}( \Tr( W_1), \dots, \Tr(W_r ) ) + o (1 ). 
\end{equation}
Since $ \cE(p, W_1, \dots, W_r) $ factorizes, up to $o(1)$,
over the blocks of $ p\delta p \delta \vee \overline{\gamma}
$, it suffices to show that if $ ( j ) $ , and respectively
$ ( k, l ) $, are blocks of $ \rho( p ) $, then
\[
\cc_1( \Tr( W_j ) ) = \sum_{ p \in \textbf{A} ( W_j ) }
\cE(p, W_j ) + o(1) ,
\]
and respectively
\[
\cc_2(\Tr(W_l), \Tr(W_k ) ) = \sum _{ \substack { p \in
    \textbf{A} ( W_l, W_k ) \\ p\delta p \delta \vee
    \overline{\gamma}_{ l, k } = \textbf { 1 } } } \cE(p, W_l, W_k ) +
o(1).
\]
The first equality follows from the definition of $ \cE(p,
W_j ) $ and equation (\ref{eq:mathcalE2}).  For the second
equality,  according to equation
(\ref{eq:mathcalE2}), we have 
\begin{align*}
\cc_1(\Tr(W_l) ) \cc_1(\Tr(W_k) ) =& \sum_{p\in\textbf{A}(
  W_ l ) } \cE(p, W_l) \cdot \sum_{p\in\textbf{A}( W_ k ) }
\cE(p, W_k) + o (1 )\\ =& \sum_{ \substack{
    p\in\textbf{A}(W_l, W_k )\\ p\delta p \delta\vee
    \overline{\gamma}_{ l, k} = ( B_l, B_k ) } } \cE(p, W_l, W_k ) + o
( 1 ).
\end{align*}
The set $ \textbf{A} (W_l, W_k ) $ is the disjoint union of
the sets $ \{ p \in \textbf{A} (W_l, W_k ): p\delta p \delta
\vee \overline{\gamma}_{ l, k } = \textbf{1}_{B_l\cup B_k}
\} $ and $\{ p \in \textbf{A} ( W_l, W_k ) : p\delta p
\delta \vee \overline{\gamma}_{ l, k } = (B_l, B_k ) \} $,
therefore
\begin{align*}
\cc_2( \Tr( W_l ),& \Tr( W_k ) )
 = E ( \Tr(W_l ) \Tr(W_k ) ) - \cc_1(\Tr(W_l) ) \cc_1(\Tr(W_k)  ) \\
=&
 \sum_{ p \in \textbf{A}( W_l, W_k ) } \cE ( p, W_l, W_k  ) - 
\sum_{ \substack{ p\in\textbf{A}(W_l, W_k )\\
 p\delta p \delta \vee \overline{\gamma}_{ l, k} = ( B_l, B_k ) } } 
\cE( p, W_l, W_k ) + o ( 1 )\\
=&
\sum_{ \substack{ p\in\textbf{A}(W_l, W_k )\\
p\delta p \delta \vee \overline{\gamma}_{ l, k} =
\textbf{1}_{B_l\cup B_k} } } \cE(p, W_l, W_k ) + o( 1 ).
\end{align*}
This proves equation (\ref{eq:the-final-step}) and thus
completes the proof.

\end{proof}

\begin{theorem}\label{thm:u2a}
If $ \cA $ has a second order limit $ t$-distribution and is
independent from $\fU$, then $ \mathfrak{U} $ and $ \cA $
are asymptotically real second order free.

\end{theorem}

\begin{proof}
According to Corollary \ref{thm:bdd}, the ensembles $ \cA $
and $ \mathfrak{U} $ satisfy equation (\ref{eq:02}), hence
we only need to prove that $ \cA \cup \mathfrak{U} $ has a
second order limit $ t $-distribution. That is, we must show
that any polynomials, $ W_1, \dots, W_R $, with complex
coefficients in elements from $ \cA \cup \cA^t \cup
\mathfrak{U} $ satisfy the conditions ($i^\prime$)--($
iii^\prime$) from Definition \ref{defn:1}(2).

From Lemma \ref{lemma:017}, each $ W_i $ has the property
\[
\Tr(W_i ) = c_{i, 0} + c_{i, 1} \Tr(W_{i, 1}) + \cdots + c_{
  i, m(i) } \Tr( W_{ i, m(i) } )
\]
for some $ c_{i, k }$, depending on $N$, and some simplified
words, $ W_{i, k}$, in $\cA \cup \cA^t$ and $\mathfrak{U}
$. Since $ \cA $ has a second order limit $t$-distribution,
it follows that the limit $ \lim_{ N \lra \infty } c_{ i, k}
$ exists and it is finite, so using the multilinearity of
the cumulant functions, so only need to show the result for
$ W_1, \dots, W_R $ as in expressions from
(\ref{eq:general-word}).

Since the hypotheses of Theorem \ref{thm:distrib} are weaker
than our present theorem, we then have that $\E(\Tr(W)) =
O(1)$; so Property (\textit{i}$'$) holds.

We have assumed that $ \cA $ has a second order limit $ t
$-distribution; by Remark \ref{remark:free} $\fU$ has a
limit distribution, so all the terms of the form
\[
\sum_{ k =1}^n [ \prod_{ i =1}^n \E(\tr( C_i \cdot
  D_{ k -i } )) + \prod_{ i =1}^n \E(\tr( C_i\cdot D_{ i - k }^t
  )) ] \}
\]
from equation (\ref{eq:02}) have a limit, hence property
(\textit{ii}$^\prime$) is satisfied.

To show that Property ($iii^\prime$) is also satisfied, let
us remember that for each $ r < j \leq R $, $ W_{j} $ is a
centered polynomial with complex coefficients in elements
from $ \cA \cup\cA^t $.  Since $ \cA $ and $ \mathfrak{U} $
satisfy equation (\ref{eq:02}), $ \displaystyle \lim_{ N
  \lra \infty } \cc_2 ( \Tr( W_i ), \Tr (W_j ) ) = 0 $
whenever $ i \leq r $ and $ r < j \leq R $, thus

\begin{eqnarray}\label{eq:012} \lefteqn{
\qquad\lim_{ N \lra \infty }\Big\{
\sum_{ \pi\in \cP_{ 1, 2 } ( R ) } 
\cc_\pi ( \Tr(W_1 ),  \dots, \Tr(W_R) ) }  \\  \notag
&-& \kern-1em
\sum_{ \pi \in P_{ 1, 2} ( r )  } 
\cc_\pi ( \Tr(W_1), \dots, \Tr(W_r ) )
\kern-1em
\sum_{ \sigma \in P_{ 1, 2} (R- r ) } \kern-1em
\cc_\sigma ( \Tr(W_{ r + 1 } ), \dots, \Tr (W_R ) )\Big\} \\ \notag
& = & 0. \\ \notag
\end{eqnarray}
Also $ \Tr(W_{ r + 1 }), \dots, \Tr(W_R ) $ have the
following properties: the first order cumulants are zero
(from the choice of $ W_1, \dots, W_R $ ); the second order
cumulants have a finite limit; and their cumulants of order
higher than 2 are $ o( 1 )$ (since the ensemble $ \cA $ has
a second order limit $ t $ -distribution). Thus we have that
\begin{eqnarray}\label{eq:013}\lefteqn{
\E \big( \Tr( W_{r + 1 }) \cdots \Tr ( W_R ) \big) }\\
& = &
\sum_{ \sigma \in \cP_{ 1, 2 } ( R - r ) } 
\cc_\sigma ( \Tr ( W_{ r + 1 } ), \dots, \Tr ( W_R ) ) +
o(1).\notag
\end{eqnarray}

\noindent 
Finally, equations (\ref{eq:mathcalE2}) and (\ref{eq:011})
give
\begin{align*}
\E(&\Tr(W_1)  \cdots \Tr(W_R))  \\
&
= \sum_{ \pi \in P_{ 1, 2 } (r ) } k_\pi ( \Tr(W_1) , \dots, \Tr(W_r ) ) \cdot
\E (\Tr(W_{ r + 1 } \cdots \Tr(W_R ) ) + o(1)
\end{align*}
so the conclusion follows from equations (\ref{eq:012}) and
(\ref{eq:013}).


\end{proof}

Since the ensemble generated by the identity matrices has
a second order limit $ t $-distribution, Theorem \ref{thm:u2a}
implies the following corollary.

\begin{corollary}\label{cor:dut}
The Haar distributed random unitary $ U $ has a second order
limit $ t $-distribu\-tion.
\end{corollary}

An immediate consequence of Theorem \ref{thm:8}, Theorem
\ref{thm:u2a} from above and Corollary 21 from \cite{mss} is
the following corollary.

\begin{corollary} 
Suppose that, $ \cA $ is an ensemble of random matrices with
a second order limit $ t $-distribution, that $ \mathcal{O}
$ is an ensemble of Haar distributed random orthogonal
matrices, that $ \mathcal{U}$ is an ensemble of Haar
distributed random unitary matrices, and furthermore suppose
that $ \cA$, $ \mathcal{O} $, and $ \mathcal{U} $ are
independent. Then $ \mathcal{O} $ and $ \mathcal{U }$ are
both real and complex second order free, $ \mathcal{U }$ and
$\cA $ are complex second order free, and $ \mathcal{O} $
and $\cA $ are real second order free.
\end{corollary}

There is no contradiction in $ \mathcal{O} $ and $
\mathcal{U} $ being both asymptotically complex and real
second order free, since, from Remark \ref{remark:free}, all
traces from the second summation of equation (\ref{eq:02})
cancel asymptotically for Haar distributed random
unitaries. In the next section we will show that a similar
situation takes place for a larger class of random matrices.

\section{Unitarily-Invariant Random Matrices}

\begin{theorem}\label{thm:4-1}
Let $ \cA_1 $ and $ \cA_2 $ be two independent ensembles of
random matrices such that $ \cA_1 $ is unitarily invariant
and has the bounded cumulants property and $ \cA_2 $ has the
$ t$-bounded cumulants property. Then the ensemble $ \cA_1
\cup \cA_2 $ has the $ t $-bounded cumulants property.
\end{theorem}

\begin{proof}
We shall start with an ad-hoc refinement of Lemma
\ref{lemma:017}.  More precisely, if $ W $ is a word in
polynomials with complex coefficients in elements from $
\cA_2 \cup \cA_2^t $ and $ {\cA_1 } \cup {\cA_1^t} $, then
 \begin{equation}\label{eq:wab}
 \Tr(W) = c_0 + c_1 \Tr(W_1) + \dots c_m \Tr(W_m )
 \end{equation}
for some $ c_0, c_1, \dots, c_m $ polynomials with complex
coefficients in expectations of normalized traces of
products of elements either from $ \cA_1 $ or from $ \cA_2
\cup \cA_2^t $, and $ W_1, \dots, W_m $ either centered
polynomials with complex coefficients in elements from
$\cA_2$, or of the form
  \[ 
   \, W_i = (UA_{1} U^*)^{ (\epsilon_1) } B_{1} \cdots
  (UA_{s}U^*)^{(\epsilon_s ) } B_{ s } 
  \]
where $ \epsilon_i \in \{ 1, -1 \} $, $ A_i$'s are centered
polynomials of elements from $ \cA_1 $ and $ B_i$'s are
either centered polynomials of elements from $ \cA_2 \cup
\cA_2^t$ or identity matrices, last situation possible only
if $\epsilon_i \not=\epsilon_{i+1}$, for $ i < s $,
respectively if $ \epsilon_1 \neq \epsilon _s $, for $ i = s
$.
 
Using the unitary invariance of $ \cA_1$, it suffices to
show (\ref{eq:wab}) for 
$ \displaystyle W = (UA_{1} U^*)^{
  (\epsilon_1) } B_{1} \cdots (UA_{n}U^*)^{(\epsilon_n ) }
B_{ n } $ 
with $ A_k $'s and $ B_k $'s arbitrary polynomials in
elements from $ \cA_1 \cup \cA_1^t $, respectively in
elements from $ \cA_2 \cup \cA_2^t $.  Similarly to Lemma
\ref{lemma:017}, we will use induction on $ (n, a, b) $,
where $ a $, respectively $ b $ are the numbers of $ A_k
$'s, respectively $ B_k $'s from the expansion of $ W $ not
satisfying the conditions from (\ref{eq:wab}).  For $ n =1 $
or $ a = b = 0 $, the result is trivial. For the induction
step, if $ a \neq 0 $, then, using traciality, we can
suppose that $ A_n $ is not centered. Then, using again the
notation $ \centre{A} = A - \E(\tr(A) ) I $,
   \begin{align*}
   W = (UA_{1} & U^*)^{ (\epsilon_1) }  B_{1} \cdots
     (U \centre{A_{n}}U^*)^{(\epsilon_n ) } B_{ n }\\
      &+ \E ( \tr(A_n )) (UA_{1} U^*)^{ (\epsilon_1) } B_{1} \cdots
           (U {A_{n-1}}U^*)^{(\epsilon_{n-1} ) } B_{ n-1 } B_{ n }
   \end{align*}
and the first term has fewer not centered $ A_k $'s, while
the second term has length strictly less than $ W $.  If $ b
\neq 0 $, then similarly we can suppose that $ B_n $ does
not satisfy the conditions from (\ref{eq:wab}). Then
  \begin{align*}
    \Tr( & W  )  = \Tr((U  A_{1} U^*)^{ (\epsilon_1) }  B_{1} \cdots
      (U {A_{n}}U^*)^{(\epsilon_n ) } \centre{B_{ n }} ) \\
       + &\E ( \tr(B_n ))
       \Tr( (U {A_{n}}U^*)^{(\epsilon_n ) }(UA_{1} U^*)^{ (\epsilon_1) } B_{1} \cdots
            (U {A_{n-1}}U^*)^{(\epsilon_{n-1} ) } B_{ n-1 })
    \end{align*}
and both the arguments of $ \Tr $ from the first term have
fewer $ B_k $'s not satisfying the conditions from
(\ref{eq:wab}) than in the expression of $ W $.
 
Since the ensembles $ \cA_1 $ and $ \cA_2 \cup \cA_2^t $
have second order bounded distributions, each $ c_k $ is
bounded in $ N $.  Therefore it suffices to show that $ W_1,
W_2. \dots, W_R $ satisfy the condition from
Definition\ref{defn:1}(1), where $ W_i $'s are defined as
follows. For $ r \leq R $ and $ 1 \leq t \leq R - r $,
\[
 W_{r + t } = B_{ M + t } 
\]
with $ B_{M + t } $ a centered polynomial in elements from $
\cA_2 $.  For $ 1 \leq s \leq r $,
\[
 W_s = ( U A_{ M_{ s - 1 } + 1 } U^\ast ) ^{ \epsilon ({ M_{ s - 1 } + 1
   } ) } B_{ M_{ s - 1 } + 1 }  \cdots 
( U A_{ M_s }U^\ast ) ^{ \epsilon( M_s ) } B_{ M_s }
\]
where given positive integers $ l_1, l_2, \dots,
l_r $, we set $ M_0 = 0 $, $ M_s = M_{ s - 1 } + l_s $, $
M_r = M $, and  $ \varepsilon_i \in \{ 1, -1 \} $, each $ A_i $ are
centered polynomials of elements from $ \cA_1$ and each $
B_i$ is either a centered polynomial of elements from $
\cA_2 $ or an identity matrix, but $ B _i $
can be an identity matrix only if $ \epsilon_i = -
\epsilon_{i + 1}$ , if $ i < M_s $, or if $ \epsilon_{M_{
s-1} + 1 } =- \epsilon_{M_s}$, if $ i = M_s $. 

To show that $ {\cA_1 } \cup \cA_2 $ has the $ t $-bounded
cumulants property, as in the proof of Theorem
\ref{thm:distrib} it suffices to show that
\begin{equation}\label{eq:028}
 \E ( \Tr(W_1 ) \cdots \Tr(W_R ) ) = O(1).
\end{equation}

Note that, for $ A = ( a_{ i, j } )_{i, j =1}^N \in \cA_1$,
the $ (i_1, i_2) $-entries of $ UAU^\ast $ and 
$ ( U A U^\ast )^t $ are, respectively, 
\begin{align*}
\left[ U A U^\ast \right]_{i_1, i_2 } 
& = \sum_{ i_{ -1}, i_{ -2 } =1}^N 
    u_{ i_ 1 , i_{ -1 } } a_{ i_{ - 1 }, i_{ -2 }} 
    \overline{ u_{ i_{ 2 } , i_{ -2 } } } \\ 
\left[ (U A U^\ast )^t \right]_{i_1, i_2 } 
& = \sum_{ i_{ -1}, i_{ -2 } =1}^N 
    u_{ i_2, i_{ -1 } } a_{ i_{ - 1 }, i_{ - 2 }} 
   \overline{ u_{ i_1, i_{ - 2 } } }
= \left[ U A U^\ast  \right]_{ i_2, i_1},
\end{align*}
\noindent henceforth
\begin{align*}
\Tr(W_s) =&
 \Tr( (U A_{ M_{ s - 1 } + 1 } U^\ast )^{ 
       \epsilon({ M_{ s - 1 } + 1 }) } 
B_{ M_{ s - 1 } + 1 }
\cdots (U A_{M_s} U^\ast )^{ \epsilon ( M_s )} B_{ M_s } )\\
&\kern-4em=
\sum_{ \bi ( s ) }
( 
\prod_{ k = M_{ s - 1 } }^{ M_s }
u_{ i_{ 2 k + 1 }, i_{ - ( 2 k + 1 ) } }
\cdot
a_{ i_{ - ( 2 k + 1 ) } , i_{ - ( 2 k + 2 ) } }^{ ( k + 1 ) }
\cdot
\overline{ u_{ i_{ 2 k + 2 }, i_{ - ( 2 k + 2 ) } } }
\cdot
b_{ i_{ \psi(k) }, i_{ \psi ( - k ) } }^{ ( k ) }
)
\end{align*}
where 
$ \bi(s) = ( i_{ 2 M_{ s - 1} + 1}, i_{ - ( 2 M_{ s - 1} +  1) }, 
\dots i_{ 2 M_s},  i_{ -2M_s } ) $ 
and the map $ \psi $ is given by
\begin{enumerate}
\item[\raise0.12em\hbox{$\centerdot$}] if
 $ M_{ s -1 } <  k \leq M_s $, 
 then 
$\displaystyle \psi( k ) = \left\{ 
\begin{array}{ l l } 2 k -1 ,  & \text{ if } \epsilon ( k ) = 1 \\
 2k, & \text{ if } \epsilon ( k ) = - 1  
\end{array}
 \right. $
\item[\raise0.12em\hbox{$\centerdot$}] if
$ M_{ s -1 } <  k  < M_s $, 
then   
$\displaystyle \psi( - k  ) = \left\{ \kern-0.5em
\begin{array}{ l l } 2 k +1 ,  & \text{ if } \epsilon ( k+ 1 ) =  1  \\
 2k+2, & \text{ if } \epsilon ( k+ 1 ) =  - 1  
\end{array}
 \right. $
\item[\raise0.12em\hbox{$\centerdot$}]  
$\displaystyle \psi( -  M_s  ) = \left\{ 
\begin{array}{ l l } 2 M_{ s - 1 } + 1 ,  & \text{ if } \epsilon ( 1 ) =   1  \\
 2 M_{ s - 1 } + 2, & \text{ if } \epsilon ( 1  ) =  - 1  .
\end{array}
 \right. $
\end{enumerate}

Let $ \cP_2^{\prime} ( 2 M ) = \{ p\in P_2( 2 M ) :  l+ p(l)
= 1  (\text{mod}\ 2) \text{ for all }\ l \in [ 2M] \}  $.
Expanding as above and applying Proposition \ref{prop:u2}, we obtain 
\begin{align*}
\E( \Tr&( W_1 ) \Tr( W_2 ) \cdots \Tr( W_r ) )\\
=&
\sum_{ \bi }
\E(
\prod_{ k = 1 }^{ M}
u_{ i_{ 2 k + 1 }, i_{ - ( 2 k + 1 ) } }
\cdot
a_{ i_{ - ( 2 k + 1 ) } , i_{ - ( 2 k + 2 ) } }^{ ( k + 1 ) }
\cdot
\overline{ u_{ i_{ 2 k + 2 }, i_{ - ( 2 k + 2 ) } } }
\cdot
b_{ i_{ \psi(k) }, i_{ \psi ( - k ) } }^{ ( k ) }
)
\\
=&
\sum_{ \bi } 
[ 
\E( 
u_{ i_1, i_{ - 1 } } \overline{ u_{ i_{ 2 }, i_{ - 2 } } } \cdots
 u_{ i_{ 2 M -1}, i_{ -2 M + 1 } }
\overline{ u_{ i_{ 2M  }, i_{ - 2 M } } } 
   )
\\
&\hspace{1cm} 
 \cdot \E(
a^{ ( 1 ) }_{ i_{ -1 }, i_{ -2 } } \cdots a^{ ( M) }_{ i_{ -2 M+ 1 }, i_{ -2M } }
   )
\cdot \E(
b^{ ( 1 ) }_{ i_{ \psi( 1 ) }, i_{ \psi ( - 1 ) } } \cdots b^{ ( M  ) }_{ i_{ \psi(  M ) }, i_{ \psi ( - M ) } }
)
     ]
\\
=& 
\sum_{ p,q \in \cP_2^{\prime} ( 2 M ) } \Phi_N (p, q ) \cdot 
\sum_{\bi = \bi \circ p \delta q \delta }
[
 \E ( a^{ ( 1 ) }_{ i_{ -1 }, i_{ -2 } } \cdots a^{ ( M) }_{ i_{ -2 M+ 1 }, i_{ -2M } } ) \\
&\hspace{6cm}
\cdot
\E (  b^{ ( 1 ) }_{ i_{ \psi( 1 ) }, i_{ \psi ( - 1 ) } } \cdots b^{ ( M  ) }_{ i_{ \psi(  M ) }, i_{ \psi ( - M ) } }  ) 
] 
\end{align*}

Since $ p, q \in P_2^\prime ( 2 M ) $, we have that $ p$
acts only on $ [ 2 M ] $, while $ \delta q \delta $ acts
only on $ [ -2 M ] $, hence the condition $ \bi = \bi \circ
p \delta q \delta $ is equivalent to $ \bi = \bi\circ p $
and $ \bi = \bi \circ \delta q \delta $.

Moreover, denoting $ l_k = i_{\phi (k ) } $, the condition $
\bi = \bi \circ p $ is equivalent to $\boldsymbol{l}
=\boldsymbol{l} \circ \phi^{-1} p \phi $.

Consider the map $ \omega : [ \pm M ] \lra [ 2 M ] $, given by
\[
\omega ( k ) = \left\{ 
\begin{array}{l l}
2k, & \text{ if } k < 0\\
-2 k + 1, & \text{ if } k > 0
\end{array}
\right.
\]
Denoting now $ j_k = i_{- \omega ( k ) } $, the condition $
\bi = \bi \circ \delta q \delta $ is equivalent to $ \bj =
\bj \circ \omega^{ -1 } q \omega $, hence we obtain
\begin{align*}
\sum_{\bi = \bi \circ p \delta q \delta } &
[
 \E ( a^{ ( 1 ) }_{ i_{ -1 }, i_{ -2 } } \cdots a^{ ( M) }_{ i_{ -2 M+ 1 }, i_{ -2M } } ) 
\cdot
\E (  b^{ ( 1 ) }_{ i_{ \psi( 1 ) }, i_{ \psi ( - 1 ) } } \cdots b^{ ( M  ) }_{ i_{ \psi(  M ) }, i_{ \psi ( - M ) } }  ) 
] 
\\
=&
\E (
\sum_{ \bj = \bj \circ \omega^{ -1 } q \omega } 
a_{ j_1, j_{ -1} }^{ ( 1 ) }  \cdots a_{ j_{ M }, j_{ -M } }^{ ( M ) }
 )
\cdot 
\E(
\sum_{  \boldsymbol{l} =\boldsymbol{l} \circ \psi^{-1} p \psi }
 b_{ l_1, l_{ - 1 } }^{ ( 1 ) }\cdots b_{ l_M, l_{ - M } }^{ ( M ) } 
     )
\end{align*}
From Lemma \ref{lemma:pi_epsilon}, there exist some
$\sigma_1 \in S_M $ and some $ \boldsymbol{\eta}\in \{ -1,
1\}^M $ such that
\[
\sum_{  \boldsymbol{l} =\boldsymbol{l} \circ \psi^{-1} p \psi }
 b_{ l_1, l_{ - 1 } }^{ ( 1 ) }\cdots b_{ l_M, l_{ - M } }^{ ( M ) } 
=\Tr_{\sigma_1 }( B_1^{(\eta_1 ) }\cdots B_M^{ ( \eta_M ) } ).
\]
From the definition of $ \psi $ and since $ p \in P_2^\prime
( M ) $, we have that $ l_k = l_{ - k } $ implies $
\epsilon_k = \epsilon_{ - k} $, so if $ ( s ) $ is a
singleton of $ \sigma_1 $, then $ B_s \neq I $.  Henceforth,
Proposition \ref{prop:ord}($ i $) gives that
\[
\E(
\sum_{  \boldsymbol{l} =\boldsymbol{l} \circ \psi^{-1} p \psi }
 b_{ l_1, l_{ - 1 } }^{ ( 1 ) }\cdots b_{ l_M, l_{ - M } }^{ ( M ) } 
     )= 
O(N^{  - \frac{M}{2}})
\]

Next, observe that $ \omega^{ -1 } q \omega ( [ M ] ) = [ -M
] $.  Indeed, if $ k > 0 $, then we have $ \omega( k ) = 1
$(mod 2), hence $ q \omega ( k ) =0$(mod 2) and $ \omega^{ -
  1 } q \omega( k ) < 0 $.  Therefore, according to Remark
\ref{rem:2}, there exists some $ \sigma_2 \in S_M $ such
that
\[
\sum_{\bj = \bj \circ \omega^{ -1 } q \omega } 
a_{ j_1, j_{ -1} }^{ ( 1 ) } \cdots a_{ j_{ M }, j_{ -M } }^{ ( M ) }
=\Tr_{\sigma_2 } ( A_1 \cdots A_M ).
\] 
and applying Proposition \ref{prop:ord}($ i $) we obtain
\[
\E (
\sum_{ \bj = \bj \circ \omega^{ -1 } q \omega } 
a_{ j_1, j_{ -1} }^{ ( 1 ) }  \cdots a_{ j_{ M }, j_{ -M } }^{ ( M ) }
 )
=
O(N^{ -  \frac{M}{2}}).
\]

Since $ \Phi_N ( p, q ) = O( N ^M ) $, Proposition
\ref{prop:ord}($ii$) gives equation (\ref{eq:028}), hence
the conclusion.

\end{proof}
Taking $ \cA_2 $ to consist only in scalar matrices, Theorem
\ref{thm:4-1} has the following immediate consequence.
\begin{corollary}
Let $ \cA $ be a unitarily invariant ensemble of random
matrices.  If $ \cA $ has the bounded cumulants property,
then it has the $ t $-bounded cumulants property.
\end{corollary}

Another relevant consequence of Theorem \ref{thm:4-1} is the
result below.

\begin{corollary}\label{cor:t-free}
If $ \cA $ is a ensemble of unitarily invariant random
matrices with a limit distribution and the bounded cumulants
property, then $ \cA $ and $ \cA^t = \{ A^t :\ A \in \cA \}
$ are asymptotically free. In particular, the ensemble $ \cA
\cup \cA^t $ has a limit distribution.

\end{corollary}

\begin{proof}
Taking the ensemble $ \cA_2 $ from Theorem \ref{thm:4-1} to
consist only on scalar matrices, equation (\ref{eq:028})
implies that for all $ A_1, \dots, A_m $ centered
polynomials in elements from $ \cA$,
\[ \lim_{ N \lra \infty } \E( \tr( A_1 A_2^t A_3 \cdots ) ) =  0 \]
hence the conclusion.
\end{proof}

Note that if $ \cB \subset \cA $ is an ensemble of symmetric
random matrices, then $ \cB $ is also contained in $ \cA^t
$.  This situation is discussed in Corollary \ref{cor:symm}
below. Note that in a non-commutative probability space
$(\cC, \phi)$ an element $a$ with the property that
$\phi(a^n) = 0$ for all $n \geq 1$ is free from itself. For
example a Haar unitary has this this property, but of course
it is not $*$-free from itself.

\begin{lemma}\label{lemma:symm}
Suppose that $\{ B_N\}_N $ is an ensemble of random matrices
with a limit distribution and $ \displaystyle \beta = \lim_{
  N \longrightarrow\infty } E ( \tr ( B_N) ) $.
\begin{enumerate}

\item[(\textit{i})]If $ \{ B_N \}_n $ and $ \{ B^\ast_N \}_n
  $ are asymptotically free, then $ \{ B_N - \beta I_N \}_N
  $ converges to $0$ in $ L^2 ( E \circ \tr ) $.

\item[(\textit{ii})]If $ \{ B_N \}_N $ is asymptotically
  free from itself, then $ \{ B_N - \beta I_N \}_N $
  converges to $0$ in distribution with respect to $ E \circ
  \tr $, but not necessarily in $ L^2 ( E \circ \tr ) $.

\end{enumerate}
\end{lemma}

\begin{proof}

Suppose that $ \{ B_N \}_N $ and $ \{ B^\ast_N \}_N $ are
asymptotically free. The definition implies that
\begin{align*}
\lim_{ N \longrightarrow \infty } E ( \tr ( &( B_N - \beta
I_N ) ( B_N - \beta I_N )^\ast ) ) ) \\ =& \lim_{ N
  \longrightarrow \infty } E ( \tr ( ( B_N - \beta I_N ))
\cdot \lim_{ N \longrightarrow \infty } E ( \tr ( ( B_N -
\beta I_N )^\ast ) =0.
\end{align*}

Suppose that $ \{ B_N \}_N $ is asymptotically free from
itself.  Then, for any positive integer $ m $, we have that
\begin{align*}
\lim_{ N \longrightarrow \infty } &(E \circ \tr ) ( ( B_N - \beta I_ N )^{  m + 1 } )\\
 = &
\lim_{ N \longrightarrow \infty } (E \circ \tr ) ( ( B_N - \beta I_ N )^{  m } ) \cdot 
\lim_{ N \longrightarrow \infty } (E \circ \tr ) ( ( B_N - \beta I_ N ) ) =0
\end{align*}

For the last part, take $ \{ B_N \}_N = \{ U + U^t \}_N $.
Then $ \{ B_N \}_N $ is an ensemble of unitarily invariant
random matrices and, according to Corollary
\ref{cor:t-free}, it is asymptotically free from $ \{ B_N^t
\}_N = \{ B_N \}_N $. Corollary \ref{cor:ut} gives
\begin{align*}
\lim_{ N \longrightarrow \infty } ( E \circ \tr ) ( &( U + U^t )(U+U^t)^\ast ) = \\
\lim_{ N \longrightarrow \infty } &( E \circ \tr ) (UU^\ast + U (U^t)^\ast + ( U^\ast U )^t + U^t U^\ast )
= 2\neq 0.
\end{align*}

\end{proof}

\begin{corollary}\label{cor:symm}
Suppose that $ \cA $ and $ \{ B_N \}_N $ are ensembles of
unitarily invariant random matrices having limit
distributions and the bounded cumulants property such that $
\{ B_N \}_N \subset \cA \cup \cA^t $ and let $ \displaystyle
\beta = \lim_{ N \longrightarrow \infty } E ( \tr ( B_N ) )
$.
\begin{enumerate}

\item[(\textit{i})]If $ \cA $ is closed under taking
  adjoints (i.e. $ A \in \cA $ implies $ A^\ast \in \cA $),
  then $ \{ B_N - \beta I_N \}_N $ converges to 0 in $ L^2 (
  E \circ \tr ) $.

\item[(\textit{ii})]If $ \cA $ is not closed under taking
  adjoints, then $ \{ B_N - \beta I_N \}_N $ converges to 0
  in distribution with respect to $ E \circ \tr $, but not
  necessarily in $ L^2 ( E \circ \tr ) $.

\end{enumerate}
\end{corollary}

\begin{proof}

For part (i), since taking transposes and taking adjoints
commute, it follows that if $ \cA $ is closed under taking
adjoints, then so is $ \cA^t $ hence $ \{ B_N^\ast \}_N
\subset \cA^t $ is asymptotically free (according to
Corollary \ref{cor:t-free}) from $\{ B_N \}_N \subset \cA $, and
the conclusion follows from Lemma \ref{lemma:symm}
(\textit{i}).

For part (\textit{ii}), it suffices to note that the example
from the proof of Lemma \ref{lemma:symm} (\textit{ii}) is
symmetric and unitarily invariant.

\end{proof}

\begin{theorem}\label{thm:4-2}
Let $ \cA_1 $ and $ \cA_2 $ be two independent ensembles of
random matrices such that $ \cA_1 $ is unitarily invariant
and has a second order limit distribution and $ \cA_2 $ has
a second order limit $ t $-distribution.  Then the ensemble
$ \cA_1 \cup \cA_2 $ has a second order limit $ t
$-distribution.
\end{theorem}

\begin{proof}
Since the coefficients $ c_0, c_1, \dots, c_m $ from
equation (\ref{eq:wab}) are polynomials with complex
coefficients in expectations of normalized traces of
products of elements either from $ \cA_1 $ or from $ \cA_2
\cup \cA_2^t $, it follows that under present assumptions
$\displaystyle \lim_{ N \lra \infty} c_k $ exists and it is
finite for all $ k $, hence it suffices to show that the
conditions ($ i^\prime $)-($ iii^\prime $) from Definition
\ref{defn:1}(2) hold true for $ W_1, \dots, W_R $ as in the
proof of Theorem \ref{thm:4-1}.  More precisely, For $ r
\leq R $ and $ 1 \leq t \leq R - r $,
\[
 W_{r + t } = B_{ M + t } 
\]
 with $ B_{M + t } $ a centered polynomial in elements from $ \cA_2 $.
For $ 1 \leq s \leq r $, 
\[
 W_s = ( U A_{ M_{ s - 1 } + 1 } U^\ast ) ^{ \epsilon ({ M_{ s - 1 } + 1
   } ) } B_{ M_{ s - 1 } + 1 }  \cdots 
( U A_{ M_s }U^\ast ) ^{ \epsilon( M_s ) } B_{ M_s }
\]
where given positive integers $ l_1, l_2, \dots,
l_r $, we set $ M_0 = 0 $, $ M_s = M_{ s - 1 } + l_s $, $
M_r = M $, and  $ \varepsilon_i \in \{ 1, -1 \} $, each $ A_i $ are
centered polynomials of elements from $ \cA_1$ and each $
B_i$ is either a centered polynomial of elements from $
\cA_2 $ or an identity matrix, but $ B _i $
can be an identity matrix only if $ \epsilon_i = -
\epsilon_{i + 1}$ , if $ i < M_s $, or if $ \epsilon_{M_{
s-1} + 1 } =- \epsilon_{M_s}$, if $ i = M_s $.

Property ($i^\prime$) is an immediate consequence of
equation (\ref{eq:028}), since the hypotheses of Theorem
\ref{thm:4-1} are satisfied.

Since the ensemble $ \cA_1 $ is unitarily invariant, it is
also orthogonally invariant, so, from Theorem \ref{thm:8},
the ensembles $ \cA_1 $ and $ \cA_2 $ satisfy equation
(\ref{eq:02}), hence $ k_2 ( \E ( \Tr(W_i ), \E( \Tr (W_j )
) = o(1 ) $ whenever $ i \leq r < j $. The ensemble $ \cA_2
$ has a second order limit $ t $-distribution, so property
($ii^\prime$) is proved if we show that $ \displaystyle
\lim_{ N \lra \infty } k_2 ( \E ( \Tr(W_i ), \E( \Tr (W_j )
) $ exists and it is finite for $ i, j \leq r $.  From
Corollary \ref{cor:t-free}, the ensembles $ \cA_1 $ and $
\cA_1^t $ are asymptotically free, so an alternating product
of centered polynomials in elements from $ \cA_1$,
respectively $ \cA_1^t $ is asymptotically centered, so we
can suppose that $ W_i = C_1 C_2 \cdots C_m $ and $ W_j =
D_1D_2 \cdots D_n $ where if $ k, p $ are odd then $ C_k $
and $ D_p $ are centered polynomials in elements from $
\cA_1 \cup \cA_1^t $ and if $ k, p $ are even, then $ C_k $
and $ D_p $ are centered polynomials in elements from $
\cA_2 \cup \cA_2^t $.  Equation (\ref{eq:02}) gives then:
\begin{align*}
\lim_{ N \lra \infty}  k_2 &( \E ( \Tr(W_i ), \E( \Tr (W_j ) ) - \\
& \delta_{m, n} \cdot 
\sum_{ k =1}^n [ \prod_{ l =1}^n \E(\tr( C_l  \cdot
  D_{ k -  l } )) + \prod_{ l  =1}^n \E(\tr( C_l \cdot D_{ l  - k }^t
  )) ] \} = 0
\end{align*}
therefore property ($ii^\prime $) from Definition
\ref{defn:1}(2) is proved if we show that $ \displaystyle
\lim_{ N \lra \infty } \tr ( C_k D_l ) $ exists and it is
finite for all $ k, l $.  If $ k $ and $ l $ have the same
parity, the limit exists since from the hypothesis,
respectively Corollary \ref{cor:t-free}, the ensembles $
\cA_2 \cup \cA_2^ t$ and $ \cA_1 \cup \cA_1^t $ have limit
distributions. If $ k $ and $ l $ have different parities,
then the unitary invariance of $ \cA_1 $ gives (see, for
example \cite{mss}) that $ \cA_1 \cup \cA_1^t $ and $
\cA_2\cup\cA_2^t $ are asymptotically free, so the limit is
zero.

For property ($iii^\prime$), we need to show that
\begin{equation} \label{eq:4k3}
\lim_{N\lra \infty } k_3 ( \Tr( W_{ i_1} ), \dots, \Tr( W_{ i_s } ) = 0
\end{equation}
 whenever $ s \geq 3 $ and $ 1 \leq i_k \leq R $.

If the set $\{ i_1, \dots, i_s \} $ contains both elements
less than $ r $ and greater than $ r $, note first that,
from Theorem \ref{thm:4-1}, the ensemble $ \cA_1 \cup \cA_2
$ has a second order bounded $ t $-distribution. Therefore
we can apply Proposition \ref{prop:ord}($ii$) to $ W_1,
\dots, W_R $, which implies (\ref{eq:4k3}).  If $ i_k \leq r
$ for all $ k $, then (\ref{eq:4k3}) is given by Lemma
\ref{lemma:4-preliminary_to_thm_22}.  Finally, if $ i_k > r
$ for all $ k $, then (\ref{eq:4k3}) is holds true since $
\cA_2 $ has a second order limit $ t $-distribution.
\end{proof}

Since ensembles of scalar matrices matrices have a second
order limit $ t $-distribution, Theorem \ref{thm:4-2}
implies the following corollary.
\begin{corollary}\label{cor:19}
Let $ \cA $ be a unitarily invariant ensemble of random matrices. 
If $ \cA $ has a second order limit distribution,
 then it
has a second order limit $ t$-distribution.
\end{corollary}

\begin{remark} 
Let $ \cA $ be an ensemble with a second order limit
distribution. The ensemble
\[
\widetilde{\cA} = \{ \int_{\mathcal{U}(N ) } U A U^\ast
d\,\mathcal{U}_N(U) : A \in \cA \}
\]
is unitarily invariant (from the properties of the Haar
measure) and has a second order limit distribution (since $
\cA $ has a second order limit distribution). Hence,
Corollary \ref{cor:19} implies that $ \widetilde{ \cA } $
has a second order limit $ t $-distribution.
\end{remark}

Since orthogonal matrices are unitary, any unitarily invariant
ensemble is also orthogonally invariant.  From Corollary
\ref{cor:19}, a unitarily invariant ensemble with a second
order limit distribution also has a second order limit $ t
$-distribution, hence Theorem \ref{thm:8} and Corollary
\ref{cor:19} implies the following

\begin{corollary} 
If $ \cA_1 $ and $ \cA_2 $ are two independent ensembles of
random matrices such that $ \cA_1 $ is unitarily invariant and
has a second order limit distribution and $ \cA_2 $ has a
second order limit $ t $-distribution, then $ \cA_1$ and
$\cA_2 $ are asymptotically real second order free.
\end{corollary}

\begin{remark}\label{rem:23}
Some relevant classes of unitarily invariant ensembles of
random matrices (namely complex Gaussian and Wishart) are
describ\-ed in \cite{me}, \cite{ms}, \cite{mss}. Moreover, in
\cite{mss} it is shown that a unitarily invariant ensemble of
random matrices with a second order limit distribution is
asymptotically \emph{complex} second order free from any
independent ensemble of random matrices with a second order
limit distribution. This result, together with Corollary 21
imply that unitarily invariant ensembles of random matrices
with a second order limit distribution are \emph{both real
  and complex} second order free from any independent
ensemble of random matrices with a second order limit
distribution (see also \cite{mp} Section 10). Particularly,
the following relation (under a weaker form also mentioned
in \cite{mp}, Section 10) holds true.

\begin{proposition}
If $ \cA = \{ A_N, B_N \}_N $ is an ensemble of random
matrices such that $ \displaystyle \sup_{ N \rightarrow
  \infty} | \E ( \Tr(A B ) ) | < \infty $ and $
\displaystyle \sup_{ N \rightarrow \infty} | \E ( \Tr (A)
\Tr( B ) ) |< \infty $, and $ \{ U_N \}_N $ is an ensemble
of Haar unitary matrices independent from $ \cA $, then
\[ \displaystyle \lim_{ N
  \rightarrow \infty } \E \left( \tr ( U A U^\ast ( U B
U^\ast )^t ) \right) = 0. \]
\end{proposition}

\begin{proof}
Denote $ A = ( a_{ i, j } )_{i,j = 1}^N $ and $ B = ( b_{ i,
  j })_{ i,j = 1}^N $. Then
\begin{eqnarray*} \lefteqn{
\E ( \tr ( U A U^\ast ( U B U^\ast )^t )) } \\
& = &
 \frac{1}{N} \sum_{ \substack{ 1 \leq i_k \leq N\\ 1 \leq | k | \leq 4 } }
\E(
u_{i_1, i_{-1 }} a_{ i_{ -1 } , i_{ -2 } } \overline{ u_{ i_2, i_{ -2 } } } 
u_{i_3, i_{-3 }} b_{ i_{ -3 } , i_{ -4 } } \overline{ u_{ i_4, i_{ -4 } } } )\\
& = &
 \frac{1}{N} \sum_{ \substack{ 1 \leq i_k \leq N\\ 1 \leq | k | \leq 4 } }
\E(a_{ i_{ -1 } , i_{ -2 } } b_{ i_{ -3 } , i_{ -4 } } )
\cdot
 \E ( u_{i_1, i_{-1 }}\cdot  u_{i_3, i_{-3 }} \cdot \overline{ u_{ i_2, i_{ -2 } } } \cdot \overline{ u_{ i_4, i_{ -4 } } } ).
\end{eqnarray*}

From Theorem \ref{thm:unit1}, the right-hand side summand
cancels unless\\ $( i_{ -1}, i_{ -3} ) \in \{ ( i_{ -2} ,
i_{ -4 } ), ( i_{ -4 }, i_{ -2} ) \} $. If $( i_{ -1}, i_{
  -3} ) = ( i_{ -2} , i_{ -4 } ) $, then from Remark
\ref{rem:ord} and Theorem \ref{thm:unit1}, the right-hand
side of the above equation equals
\begin{align*}
\frac{1}{N}
\E \left(\Tr(A)\Tr(B) \right )
 \sum_{ \substack { 1 \leq i_k \leq N\\ 1 \leq | k | \leq 4 } }
 \E ( u_{ i_1, j } \cdot  u_{ i_3, k }  \overline { u_{ i_2, j } } 
\cdot \overline { u_{ i_4, k } } ) = O( N^{ -1 } ).
\end{align*}
Similarly, if $( i_{ -1}, i_{ -3} ) = ( i_{ - 4 } , i_{ - 2
} ) $, then the summation equals
\begin{align*}
\frac{1}{N}
\E \left(\text{Tr}(AB) \right )
 \sum_{ \substack { 1 \leq i_k \leq N\\ 1 \leq | k | \leq 4 } }
 \E ( u_{ i_1, j } \cdot  u_{ i_3, k }  \overline { u_{ i_2, j } } \cdot \overline { u_{ i_4, k } } ) = O( N^{ -1 } ).
\end{align*}
\end{proof}

\begin{proposition}\label{prop:31}
Suppose that $ \cA $ is a unitarily-invariant ensemble of
random matrices that have a second order limit $ t
$-distribution.  Then $ \cA $ and $ \cA^t = \{ A^t : \ A \in
\cA \} $ are asymptotically real second order free.

In particular, if $ U_N $ is a $ N \times N $ Haar unitary
random matrix, then $ \{ U_N, U_N^\ast \}_N $ and $ \{ U_N^t
, \overline{U}_N \} $ are asymptotically real second order free.
\end{proposition}

\begin{proof}
The asymptotic free independence of the ensembles $ \cA $
and $ \cA^t $ is shown in Corollary \ref{cor:t-free}.

Let $ \{ U_N \}_N $ be an ensemble of Haar unitary random
matrices independent from $ \cA $ (hence from $ \cA^t $) and
denote by $ \mathfrak{U} $ the ensemble $\{ U_N, U_N^t,
U_N^\ast, \overline{U}_N \}_N $.

Since $ \cA $ is unitarily invariant, any cumulant of traces
of products of elements from $ \cA $ and $ \cA^t $ equals a
cumulant of same order of traces of products of elements
from $ \cA $, $ \cA^t $ and $ \mathfrak{U} $, which vanishes
asymptotically from Theorem \ref{thm:u2a} (i.e. from the
asymptotic real second order freeness of $ \mathfrak{U} $
and $ \cA\cup \cA^t $).

It remains to prove that $ \cA $ and $ \cA^t $ satisfy
equation (\ref{eq:02}). Let $ A_1, \dots, \ab A_n $ and $
B_1, \dots, B_m $ be centered random matrices such that if $
k $ is odd then $ A_k, B_k \in \cA $ and if $ k $ is even
then $ A_k, B_k \in \cA^t $. Then
\begin{align*}
\Tr(A_1 A_2 \cdots A_n ) & = \Tr( U A_1 U^\ast ( U^\ast )^t
A_2 U^t U A_3 \cdots )\\
& = 
\Tr (A_1 ( U^t U )^\ast A_2 U^t U A_3 \cdots )
\end{align*}
where the last factors of the last product above are $ ( U^t U
)^\ast A_n $ if $ n $ is odd, respectively $ A_n U^t U $ if
$ n $ is even.

Since $ E( \tr ( U^t U ) ) = 0 $ and $ E ( \tr ( U^t U ( U^t
U )^\ast ) ) = 1 $, applying again Theorem \ref{thm:u2a} we
obtain
\begin{align*}
\lim_{ N \rightarrow\infty} 
\Big\{\text{cov}(\Tr(A_1 &  A_2   \cdots A_n ), \Tr( B_1, B_2 \cdots B_m) )\ 
-
\\
& \delta_{m, n}\cdot \sum_{ k =1}^n  
\Big[ 
 \prod_{ i =1}^n \tr( A_i \cdot B_{ k -i } ) + \prod_{ i =1}^n \tr( A_i\cdot  B_{ i - k }^t ) 
 \Big] \Big\} =0
\end{align*}
hence the conclusion.

Moreover, from Remark \ref{rem:23}, $ \displaystyle \lim_{ N
  \rightarrow \infty} \tr( A_l \cdot B_{ k } ) = 0 $ for all
$ l $ and $ k $, therefore $ \cA $ and $ \cA^t $ are
\emph{not complex second order free}.
\end{proof}

\emph{It is an open question at this moment if there exist $
  \cA $, $ \cB $ two ensembles of random matrices with a
  second order limit $t$-distribution such that $ \cA $ and
  $\cB $ are asymptotically complex second order free, but
  not real second order free.}
\end{remark}

\end{document}